\documentclass[9pt]{article}
\usepackage[T1]{fontenc}
\usepackage{lmodern}
\usepackage[a4paper,margin=1in]{geometry}
\usepackage{amsmath,amssymb,amsthm,mathtools}
\usepackage{enumitem,booktabs,microtype}
\usepackage{cite}
\usepackage{xcolor}
\usepackage[hidelinks]{hyperref}
\allowdisplaybreaks
\numberwithin{equation}{section}
\newtheorem{theorem}{Theorem}[section]
\newtheorem{lemma}[theorem]{Lemma}
\newtheorem{proposition}[theorem]{Proposition}
\newtheorem{corollary}[theorem]{Corollary}
\theoremstyle{definition}
\newtheorem{assumption}[theorem]{Assumption}
\newtheorem{definition}[theorem]{Definition}
\newtheorem{algorithm}[theorem]{Algorithm}
\theoremstyle{remark}
\newtheorem{remark}[theorem]{Remark}
\newtheorem{example}[theorem]{Example}
\newcommand{\R}{\mathbb R}

\newcommand{\Sset}{\mathcal S}
\newcommand{\ind}{\iota}
\newcommand{\whpartial}{\widehat\partial}
\newcommand{\norm}[1]{\left\lVert #1\right\rVert}
\newcommand{\ip}[2]{\left\langle #1,#2\right\rangle}
\newcommand{\abs}[1]{\left\lvert #1\right\rvert}
\newcommand{\Gcal}{\mathcal G}
\newcommand{\Ecal}{\mathcal E}
\newcommand{\Ccal}{\mathcal C}

\newcommand{\clB}{\overline{\mathbb B}}

\newcommand{\Sone}{\mathbb S^1}
\newcommand{\h}[1]{\mathbf{#1}}

\newcommand{\nuf}{\boldsymbol\nu}
\DeclareMathOperator{\dist}{dist}
\DeclareMathOperator{\dom}{dom}
\DeclareMathOperator{\epi}{epi}
\DeclareMathOperator{\prox}{prox}
\DeclareMathOperator{\proj}{Proj}
\DeclareMathOperator{\clust}{clust}
\newcommand{\doi}[1]{\href{https://doi.org/#1}{doi:\nolinkurl{#1}}}

\title{Whole-Sequence Convergence of Variable-Smoothing Full-Splitting Methods via a Lifted Kurdyka–\L{}ojasiewicz Framework}
\author{Min Tao\\
\small School of Mathematics and National Key Laboratory for Novel Software Technology,\\[-1mm]
\small Nanjing University, Nanjing 210093, China\\[-1mm]
\small \texttt{taom@nju.edu.cn}}
\date{}
\hypersetup{
 pdftitle={Whole-Sequence Convergence of Variable-Smoothing Full-Splitting Methods via a Lifted Kurdyka–Łojasiewicz Framework},
 pdfauthor={Min Tao},
 pdfkeywords={structured fractional programming, full splitting, variable smoothing, Kurdyka--Lojasiewicz property, global convergence}
}

\begin{document}
\maketitle

\begin{abstract}
We establish whole-sequence convergence of the primal iterates of
the smoothing-based full-splitting proximal subgradient method
(S-FSPS) of Bo\c{t}, Li, and Tao (SIAM J. Optim., 35 (2025),
pp.~2623--2653) with a prescribed, nonsummable sequence of vanishing
smoothing parameters. The challenge is that each iteration uses a
different smoothed model. We address this by constructing fixed lifted
potentials, deriving compatible sufficient-decrease and relative-error
estimates, and applying a scaled Kurdyka--\L{}ojasiewicz (KL)
finite-length argument that controls the residuals caused by changes
in the smoothing parameter. Under a KL assumption on the corresponding
fixed lifted potential and a weighted summability condition on these
residuals, the primal trajectory has finite length. For power schedules
$\gamma_k=(k+k_0)^{-\beta}$, with $k_0\geq1$ and
$1/2<\beta\leq1$, the summability condition holds when the lift
exponent is sufficiently large. Consequently, the primal sequence
converges to an exact limiting lifted stationary point without imposing full-row-rank assumptions on the linear operators. As a corollary, we
establish whole-sequence convergence for a variable-smoothing
full-splitting projected-gradient method in nonconvex nonsmooth
composite optimization. Examples distinguish primal convergence from
convergence of the auxiliary dual variables and illustrate the role of
nonsummability in guaranteeing exact stationarity.
\end{abstract}


\section{Introduction}\label{sec:intro}

We consider the structured fractional program
\begin{equation}\label{eq:problem}
 \min_{x\in\Sset} F(x):=\frac{g(Ax)+h(x)}{f(Kx)},
\end{equation}
where $\Sset\subseteq\R^n$ is nonempty, convex, and compact;
$f:\R^p\to\overline{\R}:=\R\cup\{+\infty\}$ and
$g:\R^s\to\overline{\R}$ are proper, convex, and lower semicontinuous;
$A:\R^n\to\R^s$ and $K:\R^n\to\R^p$ are linear operators;
and $h$ is differentiable with a Lipschitz continuous gradient on an open
set containing $\Sset$.
The numerator and denominator are finite and positive on $\Sset$.
The full standing assumptions are stated in Section~\ref{sec:problem}.

We analyze the smoothing-based full-splitting proximal subgradient
method (S-FSPS) of Bo\c{t}, Li, and Tao~\cite{BotLiTao2025}.
Its original analysis establishes the existence of a subsequence
converging to an exact limiting lifted stationary point. The
counterexamples in~\cite{Tao2026} show that the standing assumptions
do not ensure convergence of the whole primal sequence, even with
harmonic smoothing and exact lifted stationarity of every cluster
point. We identify additional geometric and smoothing-schedule
conditions under which the primal sequence has finite length and
converges to an exact limiting lifted stationary point.

The main challenge is that the iteration uses a changing sequence of
smoothed models. We address this by incorporating the smoothing
parameter as an additional variable in fixed lifted potentials and
establishing sufficient-decrease and relative-error estimates along
the S-FSPS trajectory. These
estimates account for the change in the Moreau parameter and, together
with a scaled KL argument, yield finite primal length. The fixed
potentials also allow the zero-smoothing limit to be treated within
the same analysis.

\subsection{Related work}\label{sec:related-work}

Variational properties of smoothing families provide a foundation for
passing from smooth models to nonsmooth problems. Burke and
Hoheisel~\cite{BurkeHoheisel2017} study epi-convergence and graphical
convergence of gradients to subdifferentials for smoothing by infimal
convolution. These approximation properties motivate the separate
trajectory question considered here: how to establish whole-sequence
convergence when the smoothing parameter changes at every iteration.

The Kurdyka--\L{}ojasiewicz (KL) property provides a standard way
for upgrading subsequential information to whole-sequence convergence
in nonsmooth nonconvex optimization. Classical analyses include the
proximal alternating framework of Attouch, Bolte, Redont, and
Soubeyran~\cite{AttouchBolteRedontSoubeyran2010}, the abstract
convergence framework of Attouch, Bolte, and
Svaiter~\cite{AttouchBolteSvaiter2013}, and the PALM analysis of
Bolte, Sabach, and Teboulle~\cite{BolteSabachTeboulle2014}. Extensions
to variable metrics and relative errors were developed by Frankel,
Garrigos, and Peypouquet~\cite{FrankelGarrigosPeypouquet2015}; their
framework  permits variable descent coefficients without a
uniform upper bound and additive relative errors.
Sun~\cite{Sun1} treats inexact nonconvex nonsmooth algorithms under
perturbed sufficient-decrease and relative-error conditions.

Classical smoothing frameworks include Nesterov's method for convex
functions with an explicit max structure~\cite{Nesterov2005}, the unified
first-order framework of Beck and Teboulle~\cite{BeckTeboulle2012},
and Chen's gradient-consistent approximations for nonconvex
problems~\cite{Chen2012}.
Variable-smoothing schemes have been analyzed for convex composite
optimization~\cite{BotBohm2020} and weakly convex composite
objectives~\cite{BohmWright2021}.
More recent results include complexity guarantees for an alternating
method for coupled composite optimization~\cite{LongEtAl2026} and
subsequential stationarity for a single-loop method with a parametrized
manifold constraint~\cite{KumeYamada2026}.
These results concern objective accuracy, complexity bounds for
approximate stationarity, or subsequential stationarity; they do not
resolve whole-sequence convergence of S-FSPS with vanishing
smoothing parameters.

Whole-sequence convergence has been obtained in several specialized
smoothing settings. Sabach, Teboulle, and
Voldman~\cite{SabachTeboulleVoldman2018} prove finite length and
convergence to a critical point for a smoothed clustering model with
a fixed smoothing parameter. Bian and Chen~\cite{BianChen2020}
establish whole-sequence convergence for an exact capped-$\ell_1$
relaxation of cardinality-penalized regression without invoking the
KL property; their analysis exploits the lower-bound structure of the
regularizer, support identification, and an adaptive smoothing update.

More directly related to variable smoothing, Xu, Pong, and
Sze~\cite{XuPongSze2026} and Xu, Pong, and
Zhang~\cite{XuPongZhang2026} propose smoothing moving-balls and
smoothing extended sequential quadratic methods, respectively, for
constrained difference-of-convex optimization.
In their convex specializations,
\cite[Theorem~4.3]{XuPongSze2026} and
\cite[Theorem~4.1]{XuPongZhang2026} establish convergence of the entire
primal sequence to a solution of the original constrained problem
satisfying the KKT conditions under the corresponding constraint
qualifications.
In addition to the respective model and algorithmic assumptions,
these whole-sequence convergence results assume that the smoothing
parameters $\mu_k$ satisfy
\begin{equation}\label{sum}
\mu_k\downarrow0,\qquad
\sum_{k=\lceil N/2\rceil}^{N}\mu_k\longrightarrow\infty
\quad\text{as }N\to\infty,
\qquad
\sum_{k=0}^{\infty}\mu_k^2<\infty.
\end{equation}
For power schedules $\mu_k=c(k+k_0)^{-r}$ with $c,k_0>0$,
these conditions hold if and only if $1/2<r<1$.
In particular, the harmonic schedule $r=1$ is not covered by
these conditions.
Their convex convergence proofs use perturbed Fej\'er estimates
and summable errors.
In nonconvex settings, the results concern approximate KKT
complexity bounds or subsequential stationarity, rather than
whole-sequence convergence of the primal iterates.

Related KL techniques have also been used in settings without a
vanishing smoothing parameter. Ouyang, Liu, Pong, and
Wang~\cite{OuyangLiuPongWang2025}, for example, study the transfer of
KL exponents through fixed smooth reformulations obtained by Hadamard
parametrization of $\ell_1$-regularized models. For nonconvex nonsmooth
fractional programs, Bo\c{t}, Dao, and Li~\cite{BotDaoLi2022} study
extrapolated proximal subgradient algorithms. Bo\c{t}, Li, and
Tao~\cite{BotLiTao2025} introduce two distinct full-splitting schemes.
For their prescribed vanishing-smoothing S-FSPS algorithm, the analysis
establishes the existence of a subsequence converging to an
exact limiting lifted stationary point. Their Adaptive FSPS algorithm uses extrapolation
and backtracking, with an eventually fixed positive smoothing parameter;
under its additional hypotheses, including a KL assumption,
\cite[Theorem~6.11]{BotLiTao2025} establishes whole-sequence convergence
to an approximate lifted stationary point. The present paper proves
finite primal length and whole-sequence convergence to an exact
limiting lifted stationary point for the original S-FSPS algorithm
\cite[Algorithm~4.1]{BotLiTao2025}, under additional lifted KL and
smoothing-schedule conditions.

\subsection{Main contributions}

The key analytical task is to construct a fixed potential whose
sufficient-decrease and relative-error estimates remain compatible
with the changing smoothing models along the original S-FSPS
iteration. For $k\geq1$, the primal step uses
$\nabla g_{\gamma_{k-1}}(Ax^k)$, while the next potential value
involves $g_{\gamma_k}(Ax^{k+1})$. The relative-error estimate
must therefore control
\[
 \nabla g_{\gamma_k}(Ax^{k+1})
 -\nabla g_{\gamma_{k-1}}(Ax^k),
\]
which reflects simultaneous changes in the iterate and the
smoothing parameter. Additional difficulties arise from the
growth of the inverse stepsize as $\gamma_k\downarrow0$ and
the possible oscillation of the denominator dual selections.

Our main contributions are as follows.
\begin{enumerate}
\item \emph{Fixed potentials and compatible estimates.}
We construct two fixed lifted potentials by eliminating the
numerator dual variable through the Moreau envelope and
incorporating the smoothing parameter as a variable
$\tau=\gamma^{1/q}$, where the integer $q\geq2$ is the lift
exponent. The Fenchel-denominator potential allows nonsmooth \(f\); 
a second potential, involving fewer variables, 
uses the original denominator under the assumption that \(f\)
 has a Lipschitz continuous gradient.
For each potential, we establish compatible sufficient-decrease
and relative-error estimates along the original single-loop
full-splitting iteration. The relative-error estimate explicitly
controls the simultaneous changes in the Moreau argument and
smoothing parameter displayed above.

\item \emph{Subgradient inclusions without calmness.}
Using smooth lower-support constructions, we derive explicit
Fr\'echet-subgradient inclusions for both potentials at positive smoothing parameters and at zero-smoothing cluster points.
For the Fenchel potential, these inclusions are established
without any calmness assumption on $f^*$. In particular,
the calmness assumption used in
\cite[Lemma~6.4]{BotLiTao2025} is unnecessary for the
inclusions needed in our analysis.

\item \emph{Scaled KL analysis and finite primal length.}
We develop a scaled KL finite-length framework that accommodates
unbounded descent coefficients and controls smoothing-induced
residuals through weighted summability conditions.
Building on variable-coefficient KL convergence techniques
\cite{FrankelGarrigosPeypouquet2015}, we establish a formulation
in which the displacement sequence measures only the primal motion.
Under the corresponding KL and regularity assumptions, together
with the required summability conditions, the primal sequence has
finite length and therefore converges. Combining this finite-length
property with the subsequential stationarity result guaranteed by
nonsummable smoothing yields convergence to an exact limiting
lifted stationary point.

For definable data in a common o-minimal expansion of the real
field satisfying the corresponding assumptions, the result applies
to every power smoothing schedule
$
\gamma_k=(k+k_0)^{-\beta},\; k_0\geq1,\;
\frac12<\beta\leq1,
$
by choosing a sufficiently large integer lift exponent \(q\geq2\).
In particular, the harmonic case \(\beta=1\) is also covered.

\item \emph{Composite specialization and structural phenomena.}
We specialize the analysis to nonsmooth nonconvex composite
optimization, obtaining whole-sequence convergence for the
corresponding variable-smoothing full-splitting projected-gradient
scheme. We further provide examples illustrating that convergence
of the primal sequence does not necessarily imply convergence of
the associated dual sequence, and that exact limiting lifted
stationarity may differ from limiting stationarity of the original
fractional objective.
\end{enumerate}

As in \cite{BotLiTao2025}, no full-row-rank assumption on $A$
or $K$ is imposed. Relative to the Adaptive FSPS convergence
result in \cite[Theorem~6.11]{BotLiTao2025}, our result for
S-FSPS does not require essential strict convexity of $g$
and establishes exact rather than approximate limiting lifted
stationarity. These comparisons concern the respective
analyses of distinct iterations.
The fixed-potential analysis also complements the nonconvergence
constructions in \cite{Tao2026} by identifying geometric and
schedule conditions that ensure finite primal length and
exclude those nonconvergent S-FSPS trajectories.

\medskip
\noindent\textbf{Organization.}
Section~\ref{sec:problem} introduces the notation, preliminary
results, and standing assumptions.
Section~\ref{sec:algorithm} recalls S-FSPS and its subsequential
convergence properties.
Section~\ref{sec:potential} constructs the two fixed lifted
potentials and establishes their subgradient inclusions,
sufficient-decrease estimates, and relative-error estimates.
Section~\ref{sec:main} develops the scaled KL framework and
proves whole-sequence convergence of the primal iterates.
Section~\ref{sec:examples} presents the composite specialization,
discusses the relation to the counterexamples in \cite{Tao2026},
and gives examples clarifying the distinctions between primal
and dual convergence and between the stationarity notions.
Section~\ref{sec:conclusion} concludes the paper.
\section{Problem setting and preliminaries}\label{sec:problem}
\subsection{Notation and assumptions}\label{sec:prelim}
 The Euclidean norm is denoted by \(\|\cdot\|\), and
\(\langle\cdot,\cdot\rangle\) denotes the corresponding inner product.
Let $\overline \R:=\R\cup\{+\infty\}$.
For a function \(f:\mathbb{R}^n\to\overline{\mathbb{R}}\), its
\emph{effective domain} is defined by
$
\dom f:=\{x\in\mathbb{R}^n:f(x)<+\infty\},
$
and \(f\) is said to be \emph{proper} if \(\dom f\neq\emptyset\).
The \emph{(Fenchel) conjugate} of \(f\) is given by
$
f^*(v):=\sup_{x\in\mathbb{R}^n}
\{\langle v,x\rangle-f(x)\}.
$

Let $f:\R^d\to\overline{\R}$ be proper, and let
$\overline x\in\dom f:=\{x:f(x)<+\infty\}$. Its Fr\'echet
subdifferential at $\overline x$ is
\begin{eqnarray*}\label{eq:frechet}
 \whpartial f(\overline x)
 :=\left\{v\in\R^d:
 \liminf_{\substack{x\to\overline x\\x\ne\overline x}}
 \frac{f(x)-f(\overline x)-\langle v,x-\overline x\rangle}
 {\|x-\overline x\|}\ge0\right\}.
\end{eqnarray*}

\noindent The limiting subdifferential is defined by
\begin{eqnarray*}\label{eq:limiting}
 \partial f(\overline x)
 :=\left\{v\in\R^d:
 \begin{array}{l}
 \text{there exist }x^j\to\overline x\text{ and }v^j\to v
 \text{ such that}\\
 f(x^j)\to f(\overline x),\quad
 v^j\in\whpartial f(x^j)\text{ for every }j
 \end{array}\right\}.
\end{eqnarray*}
Consequently, $
 \whpartial f(\overline x)\subseteq\partial f(\overline x).
$
The domain of the limiting subdifferential is denoted by
$
\dom(\partial f):=\{x\in\mathbb{R}^d:\partial f(x)\neq\emptyset\}.
$

For a nonempty closed convex set $\mathcal C\subseteq\mathbb R^n$,
its \emph{indicator function} is defined by $\iota_{\mathcal C}(x)=0$
for $x\in\mathcal C$ and $\iota_{\mathcal C}(x)=+\infty$ otherwise.
The \emph{normal cone} to \(\mathcal C\) at \(x\) is denoted by
$
N_{\mathcal C}(x):=\partial\iota_{\mathcal C}(x)
$.
We denote the Euclidean projection onto $\mathcal C$ by
$\proj_{\mathcal C}$.
For a sequence \(\{x^k\}_{k\ge0}\), we denote by
$\clust\{x^k\}$
the set of its cluster points.
Given a linear operator $T:{\mathbb R}^n\to{\mathbb R}^m$, we denote by $T^*$ its adjoint and write
$\sigma_T:=\|T\|=\sup\{\|Tx\|:\|x\|=1\}$ for its operator norm. In particular, $\sigma_A=\|A\|$ and $\sigma_K=\|K\|$.
Given a vector \(x\in\mathbb{R}^n\), \(x_i\) denotes its \(i\)th component.
 We use
$
    \Sone:=\bigl\{x\in\mathbb R^2:\|x\|=1\bigr\}
    \;\text{and}\;
    \overline{\mathbb B}_{R}
    :=\bigl\{x:\|x\|\leq R\bigr\}
$
to denote the unit circle and the closed ball of radius $R$,
respectively.

Throughout the paper, we assume that \(A\neq \mathbf{0}\) and use the
following standing conditions, adapted from
\cite[Assumption~3.1]{BotLiTao2025}.

\begin{assumption}\label{ass:standing}  Throughout this paper, we assume that
\begin{enumerate}[label=(\alph*),leftmargin=2.1em]
\item ${\cal S} \subseteq {\mathbb R}^n$ is a nonempty, convex, and compact set;
\item $g$ is a proper, convex, and lower semicontinuous function;
\item $h$ is differentiable on an open set containing
the compact set $\Sset$, and its gradient is $L_{\nabla h}$-Lipschitz
continuous there;
\item $f$ is proper, convex, and lower semicontinuous,
$K(\Sset)\subseteq\operatorname{int}(\dom f)$, and
$f(K\h x)>0$ for every $\h x\in\Sset$;
\item  $\alpha:=\inf_{\h x \in {\cal S}}\{g(A{\h x})+h({\h x})\}>0$;
\item  $A({\cal S}) \subseteq \dom (\partial g)$, and
 there exists a constant $\ell>0$ such that
$
{\rm dist}(0, \partial g(A \h x)) \le \ell \mbox{ for all } \h x  \in {\cal S}.
$
\end{enumerate}
\end{assumption}

\begin{definition}\label{def:stationarity}
For problem~\eqref{eq:problem}, a point
$x\in\Sset\cap A^{-1}(\dom\partial g)
\cap K^{-1}(\dom\partial f)
$
with \(f(Kx)>0\) is called a \emph{limiting lifted stationary point} if
\begin{eqnarray}\label{eq:lifted-stationarity}
0\in
f(Kx)
\bigl(
A^*\partial g(Ax)+\nabla h(x)+\partial\iota_{\mathcal S}(x)
\bigr)
-
(g(Ax)+h(x))K^*\partial f(Kx).
\end{eqnarray}
\end{definition}
This is the stationarity notion of~\cite[Definition~3.3]{BotLiTao2025}; see also~\cite[Definition~2.3]{Tao2026}. We use ``exact'' to emphasize that~\eqref{eq:lifted-stationarity} involves the original, nonsmoothed functions.

\begin{lemma}
\label{lem:denominator-bounds}
Under Assumption~\ref{ass:standing}, the following
statements hold:
\begin{itemize}
\item[(i)] there exist constants $m,M>0$ such
that
\begin{eqnarray}\label{eq:mM}
    0<m< f(Kx)\leq M
    \qquad(x\in {\cal S});
\end{eqnarray}
\item[(ii)] there exists a constant $B_{y}$ such that
for $x\in {\cal S},\ y\in\partial f(Kx)$,
\begin{eqnarray}\label{eq:y-uniform}
    \norm y\leq B_{y}.\end{eqnarray}
    \end{itemize}
\end{lemma}
\begin{proof} First, (i) follows directly from \cite[Lemma 3.5]{
BotLiTao2025}. Part (ii) follows because a proper lower
semicontinuous convex function is continuous, and its subdifferential is
locally bounded, on the interior of its effective domain.
\end{proof}

\subsection{\texorpdfstring{Moreau-envelope properties}{Moreau-envelope properties}}
For a proper, convex, and lower semicontinuous function
\(q:\mathbb{R}^n\to\overline{\mathbb{R}}\), its \emph{proximal operator}
with parameter \(\gamma>0\) is defined by
$$
\prox_{\gamma q}(x)
:=
\arg\min_{y\in\mathbb{R}^n}
\left\{
q(y)+\frac{1}{2\gamma}\|y-x\|^2
\right\}.
$$
The corresponding \emph{Moreau envelope} is defined by
\begin{eqnarray}\label{fgamma}
q_\gamma(x)
:=
\min_{y\in\mathbb{R}^n}
\left\{
q(y)+\frac{1}{2\gamma}\|y-x\|^2
\right\}.
\end{eqnarray}
For every \(x\in\mathbb{R}^n\),
$
q_\gamma(x)
=
\left(q^*+\frac{\gamma}{2}\|\cdot\|^2\right)^*(x).
$

See, for example, \cite{RockafellarWets1998,BauschkeCombettes2017} for the
subdifferential and proximal conventions used here.

\begin{lemma}\label{lem:moreau-identities}
Let $q:\R^s\to\overline\R$ be proper, lower semicontinuous, and convex.  For every
$x\in\R^s$ and $\gamma>0$, set
\begin{eqnarray}\label{pzgamma}
    p_\gamma(x):=\prox_{\gamma q}(x),
    \qquad
    z_\gamma(x):=\frac{x-p_\gamma(x)}{\gamma}.
\end{eqnarray}
Then:
\begin{enumerate}[label=(\roman*),leftmargin=2.1em]
\item The function $q_\gamma:\R^s\to\R$ defined in~\eqref{fgamma} is continuously
differentiable, with
$\nabla q_\gamma(x)=z_\gamma(x)\in\partial q(p_\gamma(x))$.
\item With $z_{\gamma}(x)$ defined by~\eqref{pzgamma}, we have
\begin{eqnarray}\label{eq:dual-prox}
 z_\gamma(x)
 =\prox_{q^*/\gamma}\!\left(\frac{x}{\gamma}\right).
\end{eqnarray}

\item For fixed $x$, the mapping $\gamma\mapsto q_\gamma(x)$ is continuously
differentiable on $(0,\infty)$ and
\begin{eqnarray}\label{eq:gamma-derivative}
    \frac{\partial}{\partial\gamma}q_\gamma(x)
    =-\frac12\norm{z_\gamma(x)}^2.
\end{eqnarray}
\item The gradient $\nabla q_\gamma$ is $1/\gamma$-Lipschitz.
Consequently, the function $u\mapsto q_\gamma(Au)$ on $\R^n$ has a
Lipschitz continuous gradient with constant $\sigma_A^2/\gamma$.
\end{enumerate}
\end{lemma}
\begin{proof}
(i) The objective in (\ref{fgamma}) is strongly convex, so its minimizer
$p_\gamma(x)$ is unique.  Its optimality condition is
$
    0\in\partial q(p_\gamma(x))
      +\frac1\gamma\bigl(p_\gamma(x)-x\bigr),
$
which is equivalent to
$z_\gamma(x)\in\partial q(p_\gamma(x))$.  Standard properties of the Moreau
envelope give
$\nabla q_\gamma(x)=z_\gamma(x)$ and the
$1/\gamma$-Lipschitz continuity of this gradient; see, for example,
\cite[Chapter~12]{BauschkeCombettes2017}.
(ii) For the proof of \eqref{eq:dual-prox}, we refer the reader to \cite{BotBohm2020}.
(iii) Danskin's
theorem therefore yields
\[
 \frac{\partial}{\partial\gamma}q_\gamma(x)
 =-\frac{1}{2\gamma^2}\norm{x-p_\gamma(x)}^2
 =-\frac12\norm{z_\gamma(x)}^2,
\]
proving \eqref{eq:gamma-derivative}.
Finally, for every $u\in\R^n$,
\[
 \nabla\bigl(q_\gamma\circ A\bigr)(u)
 =A^*\nabla q_\gamma(Au).
\]
Therefore, $u\mapsto q_\gamma(Au)$ has a Lipschitz continuous gradient
with constant $\sigma_A^2/\gamma$.
\end{proof}

\begin{lemma}
\label{lem:uniform-moreau}
Suppose that $g$ is proper, convex, and lower semicontinuous and that Assumption \ref{ass:standing}(f) holds.
Then, for every $w\in A(\mathcal S)$ and every $\gamma>0$,
\begin{itemize}
\item[(i)]
$\|\nabla g_\gamma(w)\|\leq\ell;$
\label{eq:uniform-gradient-bound}
\item[(ii)]
$
\|\prox_{\gamma g}(w)-w\|\leq\gamma\ell$;
\label{eq:uniform-proximal-bound}
\item[(iii)]
$0\leq g(w)-g_\gamma(w)\leq\frac{\gamma\ell^2}{2}$;

\item[(iv)]
$|g(u)-g(v)|\leq\ell\|u-v\|
\;\;\text{for any }u,v\in A(\mathcal S)$;
\item[(v)]
For  every $x\in\Sset$ and every $0<\gamma_1\leq\gamma_2$,
$
 0\leq g_{\gamma_1}(Ax)-g_{\gamma_2}(Ax)
 \leq\frac{\ell^2}{2}(\gamma_2-\gamma_1)$.

\end{itemize}
\end{lemma}

\begin{proof}
Parts (i)--(iv) follow from \cite[Lemma~2.4]{Tao2026}.
For part (v), integrating
\eqref{eq:gamma-derivative} from $\gamma_1$ to $\gamma_2$ and using
part~(i) yields
$
 g_{\gamma_1}(Ax)-g_{\gamma_2}(Ax)
 =\frac12\int_{\gamma_1}^{\gamma_2}
 \norm{\nabla g_t(Ax)}^2\,dt
 \leq\frac{\ell^2}{2}(\gamma_2-\gamma_1).
$
\end{proof}

\begin{lemma}
\label{lem:moreau-parameter-gradient}
For \(u=Ax\) with \(x\in{\cal  S}\) and \(0<\gamma\leq\widehat\gamma\),
\begin{eqnarray}\label{eq:moreau-parameter-gradient}
  \|\nabla g_\gamma(u)-\nabla g_{\widehat\gamma}(u)\|
  \leq
  \ell\frac{\widehat\gamma-\gamma}{\widehat\gamma}
  \leq \ell\frac{\widehat\gamma-\gamma}{\gamma}.
\end{eqnarray}
\end{lemma}

\begin{proof}
Let $p=\prox_{\gamma g}(u)$ and $\widehat p=\prox_{\widehat\gamma g}(u)$.
Then,
\[
  p=u-\gamma z,\qquad
  \widehat p=u- {\widehat\gamma} \widehat{ z},
\]
and
\[
  z=\nabla g_\gamma(u),\qquad
  \widehat z=\nabla g_{\widehat\gamma}(u).
\]

\noindent Lemma \ref{lem:moreau-identities}(i) gives \(z\in\partial g(p)\) and
\(\widehat z\in\partial g(\widehat p)\).  Monotonicity of \(\partial g\)
therefore yields
\begin{align*}
  0
  &\leq \langle z-\widehat z,p-\widehat p\rangle=-\widehat\gamma\|z-\widehat z\|^2
    +(\widehat\gamma-\gamma)\langle z-\widehat z,z\rangle.
\end{align*}
Since \(\|z\|\leq\ell\), division by
\(\widehat\gamma\|z-\widehat z\|\) when \(z\neq\widehat z\) proves
\eqref{eq:moreau-parameter-gradient}; the case of
$z=\widehat z$ is immediate.
\end{proof}

\subsection{The KL property}
\begin{definition}\label{def:KL}
A proper lower semicontinuous function $f:\R^d\to\overline \R$ has the
Kurdyka--\L{}ojasiewicz property at $\overline w\in\dom\partial f$ if
there exist a neighborhood $U$ of $\overline w$, $a>0$, and a continuous
concave function $\varphi:[0,a)\to[0,\infty)$ such that
$\varphi(0)=0$, $\varphi\in C^1(0,a)$, $\varphi'>0$, and
\begin{eqnarray}\label{eq:KL-def}
 \varphi'(f(w)-f(\overline w))\dist(0,\partial f(w))\geq1
\end{eqnarray}
whenever $w\in U$ and $f(\overline w)<f(w)<f(\overline w)+a$.
It is a KL function if it has this property throughout $\dom\partial f$.
\end{definition}

\begin{lemma}[Uniformized KL inequality]\label{lem:uniform-KL}
Let $f$ be proper and lower semicontinuous and let
$\Omega\subseteq\dom\partial f$ be compact. If $f$ is constant on
$\Omega$ and has the KL property at every point of $\Omega$, then
there are $\varepsilon,a>0$ and a concave desingularizing function
$\varphi$ such that~\eqref{eq:KL-def}, with $f(\overline w)=f(\Omega)$,
holds whenever $\dist(w,\Omega)<\varepsilon$ and
$f(\Omega)<f(w)<f(\Omega)+a$.
\end{lemma}
\begin{proof}
For the proof, we refer the reader to \cite[Lemma~6]{BolteSabachTeboulle2014}.
\end{proof}

\section{The S-FSPS algorithm and subsequential convergence}\label{sec:algorithm}
Let $\{\gamma_k\}_{k\geq0}$ be a nonincreasing sequence of positive numbers. Set
\begin{eqnarray}\label{eq:delta}
 \delta_k:=\chi\left(L_{\nabla h}+\frac{\sigma_A^2}{\gamma_k}\right),
 \qquad \chi>1.
\end{eqnarray}
We assume that
\begin{eqnarray}\label{eq:gamma-basic}
 \gamma_k\downarrow0,\;\;\mbox{and}\;\; \sum_{k=0}^{\infty}\gamma_k=\infty.
\end{eqnarray}

\noindent For $\gamma>0$, define
\begin{eqnarray}\label{eq:Psi}
 \Psi(x,z;\gamma):=\ip{Ax}{z}-g^*(z)+h(x)-\frac\gamma2\norm z^2.
\end{eqnarray}
\begin{algorithm}[S-FSPS~\cite{BotLiTao2025}]\label{alg:SFSPS}
Choose $x^0\in\Sset$, $z^0\in\R^s$, and $\theta_0>0$. Let $\{\gamma_k\}_{k\ge 0}$ satisfy~\eqref{eq:gamma-basic}, and let
$\delta_k$ be defined by~\eqref{eq:delta}. For $k\geq0$,
perform
\begin{align}
 y^{k+1}&\in\partial f(Kx^k),\label{eq:algorithm-y}\\
 x^{k+1}&=\proj_{\Sset}\!\left(x^k-
 \frac{A^*z^k+\nabla h(x^k)-\theta_kK^*y^{k+1}}{\delta_k}\right),
 \label{eq:algorithm-x}\\
 z^{k+1}&=\prox_{g^*/\gamma_k}(Ax^{k+1}/\gamma_k),
 \label{eq:algorithm-z}\\
 \theta_{k+1}&=\frac{\Psi(x^{k+1},z^{k+1};\gamma_k)}{f(Kx^{k+1})},
 \label{eq:algorithm-theta}
\end{align}
where $\Psi$ is defined in (\ref{eq:Psi}).
\end{algorithm}
All updates are well defined. Lemma~\ref{lem:moreau-identities} gives,
for $k\geq1$,
\begin{eqnarray*}\label{eq:moreau-recursion}
z^k=\nabla g_{\gamma_{k-1}}(Ax^k),\qquad
 \Psi_k(x):=g_{\gamma_{k-1}}(Ax)+h(x),\qquad
 \theta_k=\frac{\Psi_k(x^k)}{f(Kx^k)}.
\end{eqnarray*}
The restriction $k\geq1$ is necessary because $z^0$ is arbitrary.
We will use
\begin{eqnarray*}\label{eq:increments}
d_k:=x^{k+1}-x^k,\quad s_k:=\norm{d_k},\quad
 \Delta\gamma_k:=\gamma_{k-1}-\gamma_k,\quad
 \rho_k:=\frac{\Delta\gamma_k}{\gamma_{k-1}}\quad(k\geq1)
\end{eqnarray*}
and
\begin{eqnarray*}\label{eq:D-definition}
\nuf(x,y):=\ip{Kx}{y}-f^*(y),\qquad
 D_{k+1}:=\nuf(x^{k+1},y^{k+1}).
\end{eqnarray*}
The function $\nuf$ is understood to be $-\infty$ when $f^*(y)=+\infty$.
Equation (\ref{eq:algorithm-y}) gives the identity
\begin{eqnarray}\label{eq:Dhat}
 D_{k+1}=f(Kx^k)+\ip{Kd_k}{y^{k+1}}.
\end{eqnarray}

The following theorem collects the subsequential convergence
conclusions of~\cite[Theorem~2.6]{Tao2026}, together with an explicit
quotient-descent estimate and eventual bounds on the Fenchel denominator.
The conclusions recalled from that theorem remain valid when $K=0$
or $L_{\nabla h}=0$, as permitted here. We explain this extension
in the proof.

\begin{theorem}
\label{PriTheo2R}
Suppose Assumption \ref{ass:standing} holds. Define ${\cal V}^k := (x^k,y^k,z^k), \; k\geq 1$. Let $\Omega$ be the set of  cluster points of the sequence $\{{\cal V}^k\}_{k\ge 1}$ generated by Algorithm \ref{alg:SFSPS}.  Then the following statements hold:
\begin{itemize}
\item[(i)] For every $k \geq 1$, the following inequality holds:
\begin{eqnarray*}
& \ \Psi_{k+1}({ x}^{k+1})
+ \theta_k \left[f(K{ x}^k)-\nuf({ x}^{k+1},{ y}^{k+1})\right]
\leq \ \Psi_k({ x}^{k})
-\frac{\chi-1}{2\chi}\delta_k s_k^2 + \Xi^{k+1},
\end{eqnarray*}
 where
$
\Xi^{k+1} := \frac{\gamma_{k-1}-\gamma_{k}}{2}\|{ z}^{k+1}\|^2 \geq 0.
$

\item[(ii)] The sequence $\{{\cal V}^k\}_{k\ge 1}$ is bounded.
\item[(iii)] There exists an index $N_0\geq1$ such that $\theta_k\ge 0$ for all $k\geq N_0$.
\item[(iv)] There is an index $N_1\geq N_0$ such that $\theta_k>0$ for every
$k\geq N_1$.
For every $k\geq N_1$,
\begin{eqnarray}\label{eq:theta-quasidescent}
 \theta_{k+1}
 \leq \theta_k-\frac{\chi-1}{2\chi M}\delta_k s_k^2
 +\frac{\ell^2}{2m}(\gamma_{k-1}-\gamma_{k}).
\end{eqnarray}
Moreover, $\lim_{k\to+\infty}\theta_k=\overline{\theta}$
for some $\overline{\theta}\geq0$.
\item[(v)] We have $\liminf\limits_{k\to+\infty}\delta_{k}s_k=0.$
\end{itemize}
\begin{itemize}
\item[(vi)]
{For every $(\overline{ x},{\overline{ y}},{\overline{ z}}) \in \Omega$, we have }
$F(\overline x)+\iota_{\cal S}(\overline x)={\overline{ \theta}},$
{where $\overline{\theta}$ is defined in part~(iv)}.
\item[(vii)]
Let \(\{{ x}^{k_j}\}_{j\ge 0}\) be a subsequence of $\{x^k\}_{k\ge 0}$ such that
$
\lim\limits_{j \to +\infty} \delta_{k_j} s_{k_j} = 0$.  Then every cluster point of $\{x^{k_j}\}_{j\ge 0}$ is a limiting lifted stationary point of \eqref{eq:problem}.
\item[(viii)]  $
    \lim_{k\to+\infty}s_k=0,$
    and
$\displaystyle \lim_{k\to+\infty}\frac{\nuf(x^{k+1},y^{k+1})}{f(K{x}^{k})}=1.$ 
Furthermore, there exists an index $N_2\geq N_1$ such that
\begin{eqnarray}\label{K1}
0<m\le \nuf(x^k,y^k) \le f(K{ x}^k)\le M \quad \forall k \geq N_2,
\end{eqnarray}
where $m$ and $M$ are the bounds from Lemma~\ref{lem:denominator-bounds}.
\end{itemize}
\end{theorem}

\begin{proof}
Since $A\ne0$, the definition~\eqref{eq:delta} gives
$$
 \delta_k\geq\frac{\chi\sigma_A^2}{\gamma_0}>0,\;\;
 \frac1{\delta_k}
 =\frac{\gamma_k}{\chi(L_{\nabla h}\gamma_k+\sigma_A^2)}
 \geq\frac{\gamma_k}{\chi(L_{\nabla h}\gamma_0+\sigma_A^2)}.
$$
Thus $\sum_k\gamma_k=\infty$ implies $\sum_k1/\delta_k=\infty$.
These bounds remain valid when $L_{\nabla h}=0$. Moreover, neither
the descent argument nor the stationarity passage in
\cite[Theorem~2.6]{Tao2026} requires $K\ne0$; in particular, no
division by $\norm K$ is used.
We prove~\eqref{eq:theta-quasidescent} and part (viii) below,
and also give the argument for part (v).
Lemma~\ref{lem:uniform-moreau} gives
\[
 \Psi_k(x^k)\geq\alpha-\frac{\ell^2}{2}\gamma_{k-1}
 \geq\frac\alpha2
\]
for all sufficiently large $k$, so $\theta_k>0$.
Using $\Psi_k(x^k)=\theta_k f(Kx^k)$, the inequality in part (i),
and $\|z^{k+1}\|\leq\ell$, we obtain
\begin{eqnarray}\label{eq:basic-residual-descent}
 \Psi_{k+1}(x^{k+1})-\theta_k\nuf(x^{k+1},y^{k+1})
 \leq-\frac{\chi-1}{2\chi}\delta_ks_k^2
       +\frac{\ell^2}{2}(\gamma_{k-1}-\gamma_k).
\end{eqnarray}
Since $\nuf(x^{k+1},y^{k+1})\leq f(Kx^{k+1})$ and $\theta_k>0$,
dividing by $f(Kx^{k+1})$ and using the bounds
$m<f(Kx^{k+1})\leq M$ yields~\eqref{eq:theta-quasidescent}.

To prove (viii), set
$$
 T_k:=\theta_k+\frac{\ell^2}{2m}\gamma_{k-1},
 \qquad a:=\frac{\chi-1}{2\chi M}.
$$
Equation~\eqref{eq:theta-quasidescent} gives
$T_{k+1}\leq T_k-a\delta_ks_k^2$ for all sufficiently large $k$.
The sequence $T_k$ is eventually bounded below by zero.
Telescoping therefore yields $\sum_k\delta_ks_k^2<\infty$.
Since $\inf_k\delta_k>0$, we obtain $s_k\to0$.
To verify part (v), suppose that $\liminf_k\delta_ks_k>0$.
Then, for some $c>0$ and all sufficiently large $k$,
$\delta_ks_k^2\geq c^2/\delta_k$, contradicting
$\sum_k\delta_ks_k^2<\infty$ and $\sum_k1/\delta_k=\infty$.
Hence $\liminf_k\delta_ks_k=0$.
By~\eqref{eq:Dhat},
$$
 \left|\frac{D_{k+1}}{f(Kx^k)}-1\right|
 \leq\frac{\sigma_KB_y}{m}s_k\longrightarrow0.
$$
Lemma \ref{lem:denominator-bounds}(ii), \eqref{eq:Dhat}, and the fact that
$f\circ K$ is $\sigma_KB_y$-Lipschitz on $\Sset$ imply that
\begin{eqnarray}\label{add}
 |D_{k+1}-f(Kx^{k+1})|
 \leq 2\sigma_KB_y s_k\longrightarrow0.
\end{eqnarray}
Moreover, the Fenchel inequality gives
\[
 D_{k+1}\leq f(Kx^{k+1})\leq M.
\]
Let
\[
 m_f:=\min_{x\in\Sset} f(Kx).
\]
By compactness and Lemma~\ref{lem:denominator-bounds}(i), $m_f>m$.
Hence, by~\eqref{add}, for all sufficiently large $k$,
\[
 D_{k+1}\geq f(Kx^{k+1})-\frac{m_f-m}{2}
 \geq\frac{m_f+m}{2}>m.
\]
This proves the bounds~\eqref{K1}.

\end{proof}
\section{Approach to global convergence}\label{sec:potential}
Fix an integer $q\geq2$, choose $R>B_y$, and choose $\overline\tau>0$ and
$\lambda$ so that
$
 \ell^2\overline\tau^q\leq\alpha,\;
 \lambda\geq\lambda_0=\frac{\ell^2}{2m}.
$
For $x\in\Sset$ and $\tau\geq0$, define
\begin{eqnarray*}\label{eq:G-lift}
\Gcal_q(x,\tau):=
 \begin{cases}
  g_{\tau^q}(Ax)+h(x),&\tau>0,\\
  g(Ax)+h(x),&\tau=0.
 \end{cases}
\end{eqnarray*}
Items~(iii)--(iv) of Lemma~\ref{lem:uniform-moreau} imply that
$\Gcal_q$ is continuous on $\Sset\times[0,\overline{\tau}]$ and
\begin{eqnarray*}\label{eq:G-positive}
\Gcal_q(x,\tau)\geq\alpha/2
 \qquad(x\in\Sset,\ 0\leq\tau\leq\overline\tau).
\end{eqnarray*}

For every $(x,\tau)\in\Sset\times(0,\overline\tau)$, the defining
expression for $\Gcal_q$ extends to a $C^1$ function on a neighborhood
of $(x,\tau)$, with
\begin{eqnarray}\label{eq:G-derivatives}
\nabla_x\Gcal_q(x,\tau)=A^*\nabla g_{\tau^q}(Ax)+\nabla h(x),
 \qquad
 \partial_\tau\Gcal_q(x,\tau)
 =-\frac q2\tau^{q-1}\norm{\nabla g_{\tau^q}(Ax)}^2.
\end{eqnarray}
We next consider two settings and derive subgradient bounds for the
corresponding potential functions.
\subsection{The Fenchel potential}
For comparison with earlier subdifferential formulas, we recall the
following notion of relative calmness. This condition is not assumed
in the analysis below.
We say that $f^*$ is \emph{calm at $y$ relative to its effective domain}
if $y\in\dom f^*$ and there exist $L_y\geq0$ and $\varepsilon_y>0$
such that
\begin{eqnarray}\label{eq:relative-calmness}
 |f^*(v)-f^*(y)|\leq L_y\norm{v-y}
 \quad\text{for every }v\in\dom f^*
 \text{ with }\norm{v-y}<\varepsilon_y.
\end{eqnarray}
This is the convention in~\cite[Definition~2.1]{LiShenZhangZhou2022}.
The constants may depend on $y$. Requiring the same estimate for
every $v$ in a full neighborhood is stronger: it requires $f^*$
to be finite throughout that neighborhood. The latter convention
is used in~\cite[Definition~2 and Lemma~2]{HTX}.

Define the truncated domain
\begin{eqnarray}\label{eq:C-definition}
 \Ccal_q^{(1)}:=\left\{(x,y,\tau):
 x\in\Sset,\ y\in\dom f^*,\ \norm y\leq R,\ 0\leq\tau\leq\overline\tau,
 \ \nuf(x,y)\geq m/2\right\}
\end{eqnarray}
and the extended-real-valued function
\begin{eqnarray}\label{eq:E-definition}
\Ecal_q^{(1)}(x,y,\tau):=
 \begin{cases}
 \displaystyle\frac{\Gcal_q(x,\tau)+\iota_{\mathcal S}(x)}{\nuf(x,y)}+\lambda\tau^q,
       &(x,y,\tau)\in\Ccal_q^{(1)},\\[1ex]
 +\infty,&\text{otherwise}.
 \end{cases}
\end{eqnarray}

\begin{proposition}
\label{prop:E-subgradient}
The set $\Ccal_q^{(1)}$ is nonempty and compact, and
$\Ecal_q^{(1)}$ is proper and lower semicontinuous.
\end{proposition}
\begin{proof}
Take $x_0\in\Sset$. Since
$Kx_0\in\operatorname{int}(\dom f)$, choose
$y_0\in\partial f(Kx_0)$. Then
\[
 \nuf(x_0,y_0)=f(Kx_0)>m,
 \qquad
 \norm{y_0}\leq B_y<R.
\]
Thus $(x_0,y_0,0)\in\Ccal_q^{(1)}$, proving that
$\Ccal_q^{(1)}$ is nonempty and that $\Ecal_q^{(1)}$ is proper.
With the convention $\nuf(x,y)=-\infty$ when $f^*(y)=+\infty$,
the function $\nuf$ is upper semicontinuous. Consequently,
$\Ccal_q^{(1)}$ is a closed subset of the compact set
$\Sset\times\clB_R\times[0,\overline\tau]$, and hence is compact.
Let $U_j=(x_j,y_j,\tau_j)\to U=(x,y,\tau)$. If
$\liminf_{j\to\infty}\Ecal_q^{(1)}(U_j)=+\infty$,
the lower-semicontinuity inequality is immediate.
Otherwise, pass to a subsequence, without relabeling, such that
$\Ecal_q^{(1)}(U_j)<+\infty$ for every $j$ and these values
converge to the original lower limit.
Then $U_j\in\Ccal_q^{(1)}$, and the closedness of
$\Ccal_q^{(1)}$ gives $U\in\Ccal_q^{(1)}$.
Since
$
 \Gcal_q(x_j,\tau_j)\to\Gcal_q(x,\tau)>0,
 \;
 \nuf(x_j,y_j)\geq\frac m2,
$
and
$
 \limsup_{j\to\infty}\nuf(x_j,y_j)\leq\nuf(x,y),
$
we obtain
$
 \liminf_{j\to\infty}
 \frac{\Gcal_q(x_j,\tau_j)}{\nuf(x_j,y_j)}
 \geq
 \frac{\Gcal_q(x,\tau)}{\nuf(x,y)}.
$
Because $\lambda\tau_j^q\to\lambda\tau^q$, it follows that
$
 \liminf_{j\to\infty}\Ecal_q^{(1)}(U_j)
 \geq\Ecal_q^{(1)}(U).
$
Thus $\Ecal_q^{(1)}$ is lower semicontinuous.
\end{proof}

\begin{proposition}[Subgradients of the Fenchel potential]
\label{prop:E-subgradient2}
Suppose Assumption~\ref{ass:standing} holds.
\begin{itemize}
\item[(i)] Let $(x,y,\tau)\in\Ccal_q^{(1)}$ satisfy
$0<\tau<\overline\tau$ and $D:=\nuf(x,y)>m/2$. Set
\[
 G:=\Gcal_q(x,\tau),\qquad r:=G/D,\qquad
 z_\tau(x):=\nabla g_{\tau^q}(Ax).
\]
Then, for every $\zeta\in\partial f^*(y)$ and
$n\in N_{\Sset}(x)$,
\begin{equation}\label{eq:E-subgradient}
 \begin{pmatrix}
 \displaystyle\frac{A^*z_\tau(x)+\nabla h(x)-rK^*y}{D}+n\\[1.5ex]
 \displaystyle\frac{G}{D^2}(\zeta-Kx)\\[1.5ex]
 \displaystyle q\tau^{q-1}
       \left(\lambda-\frac{\norm{z_\tau(x)}^2}{2D}\right)
 \end{pmatrix}
 \in\whpartial\Ecal_q^{(1)}(x,y,\tau)
 \subseteq\partial\Ecal_q^{(1)}(x,y,\tau).
\end{equation}
\item[(ii)] Let $x\in\Sset$, $y\in\partial f(Kx)$,
and $z\in\partial g(Ax)$. At $\tau=0$, set
\[
 G:=g(Ax)+h(x),\qquad
 D:=f(Kx)=\nuf(x,y),\qquad r:=G/D.
\]
Then, for every $n\in N_{\Sset}(x)$,
\begin{equation}\label{eq:E-zero-subgradient2}
 \begin{pmatrix}
 \displaystyle\frac{A^*z+\nabla h(x)-rK^*y}{D}+n\\[1.5ex]
 0\\[1.5ex]0
 \end{pmatrix}
 \in\whpartial\Ecal_q^{(1)}(x,y,0)
 \subseteq\partial\Ecal_q^{(1)}(x,y,0).
\end{equation}
In particular, $(x,y,0)\in\dom\partial\Ecal_q^{(1)}$.
\end{itemize}
\end{proposition}
\begin{proof}
(i) Fix $\zeta\in\partial f^*(y)$. The convex subgradient inequality gives
$f^*(v)\geq f^*(y)+\ip{\zeta}{v-y}$, and therefore
\begin{align}
 \nuf(u,v)
 &\leq D+\ip{K^*y}{u-x}+\ip{Kx-\zeta}{v-y}
                   +\ip{K(u-x)}{v-y}\notag\\
 &=:B_\zeta(u,v).
 \label{eq:denominator-support}
\end{align}
Since $B_\zeta(x,y)=D>0$ and $\Gcal_q(x,\tau)>0$, both
$B_\zeta(u,v)$ and $g_{t^q}(Au)+h(u)$ are positive on a sufficiently
small open neighborhood of $U_0=(x,y,\tau)$ on which $t>0$.
On this neighborhood, define the $C^1$ function
\[
 \Phi_\zeta(u,v,t)
 :=\frac{g_{t^q}(Au)+h(u)}{B_\zeta(u,v)}+\lambda t^q.
\]
For $U=(u,v,t)$ in this neighborhood and in the effective domain
of $\Ecal_q^{(1)}$,
we have $u\in\Sset$, so $\iota_{\Sset}(u)=0$.
Positivity and $\nuf(u,v)\leq B_\zeta(u,v)$ therefore give
\[
 \Ecal_q^{(1)}(U)
 =\frac{g_{t^q}(Au)+h(u)}{\nuf(u,v)}+\lambda t^q
 \geq\Phi_\zeta(U).
\]
The inequality also holds outside the effective domain, where
$\Ecal_q^{(1)}(U)=+\infty$, and equality holds at $U_0$.
Thus $\Phi_\zeta$ is a smooth lower support of $\Ecal_q^{(1)}$
at $U_0$. Differentiation using~\eqref{eq:G-derivatives} and
\eqref{eq:denominator-support} yields
\[
 \nabla\Phi_\zeta(U_0)=
 \begin{pmatrix}
 \displaystyle\frac{A^*z_\tau(x)+\nabla h(x)-rK^*y}{D}\\[1.5ex]
 \displaystyle\frac{G}{D^2}(\zeta-Kx)\\[1.5ex]
 \displaystyle q\tau^{q-1}
       \left(\lambda-\frac{\norm{z_\tau(x)}^2}{2D}\right)
 \end{pmatrix}.
\]
To incorporate the constraint normal, fix $n\in N_{\Sset}(x)$
and set $v_n:=\nabla\Phi_\zeta(U_0)+(n,0,0)$.
Since $\Phi_\zeta$ is $C^1$, its first-order remainder satisfies
\begin{equation}\label{eq:support-remainder}
 R_\Phi(U):=\Phi_\zeta(U)-\Phi_\zeta(U_0)
                 -\ip{\nabla\Phi_\zeta(U_0)}{U-U_0},
 \qquad
 \frac{R_\Phi(U)}{\norm{U-U_0}}\longrightarrow0.
\end{equation}
For every sufficiently close $U$ in the effective domain,
$u\in\Sset$ and hence $\ip n{u-x}\leq0$. Consequently,
\[
 \begin{aligned}
 &\Ecal_q^{(1)}(U)-\Ecal_q^{(1)}(U_0)
                -\ip{v_n}{U-U_0}\\
 &\qquad\geq R_\Phi(U)-\ip n{u-x}
 \geq R_\Phi(U).
 \end{aligned}
\]
Outside the effective domain the same inequality holds trivially.
Dividing by $\norm{U-U_0}$ and taking the lower limit as
$U\to U_0$, $U\neq U_0$, proves
$v_n\in\whpartial\Ecal_q^{(1)}(U_0)$ by the definition of the
Fr\'echet subdifferential. This is~\eqref{eq:E-subgradient}.
The inclusion in the limiting subdifferential follows from
$\whpartial\Ecal_q^{(1)}\subseteq\partial\Ecal_q^{(1)}$.

(ii) The denominator bounds and $R>B_y$ imply
$(x,y,0)\in\Ccal_q^{(1)}$. For $t>0$, the Moreau definition and
$z\in\partial g(Ax)$ yield
\begin{align}
 g_{t^q}(Au)
 &\geq\inf_w\left\{
 g(Ax)+\ip z{w-Ax}+\frac{\norm{w-Au}^2}{2t^q}\right\}\notag\\
 &=g(Ax)+\ip{A^*z}{u-x}-\frac{t^q}{2}\norm z^2.
 \label{eq:zero-envelope-support}
\end{align}
The last equality holds because the minimum is attained at $w=Au-t^qz$.
With $g_0:=g$, inequality~\eqref{eq:zero-envelope-support} also holds
at $t=0$ by the convex subgradient inequality.
Define
\[
 P(u,t):=g(Ax)+\ip{A^*z}{u-x}+h(u)-\frac{t^q}{2}\norm z^2.
\]
Then $\Gcal_q(u,t)\geq P(u,t)$ for $u\in\Sset$ and $t\geq0$,
and $P(x,0)=G>0$. Since $y\in\partial f(Kx)$, we have
$Kx\in\partial f^*(y)$. Taking $\zeta=Kx$ in
\eqref{eq:denominator-support} gives
\[
 \nuf(u,v)\leq
 D+\ip{K^*y}{u-x}+\ip{K(u-x)}{v-y}=:B(u,v).
\]
Here $B(x,y)=D>0$ and $\nabla_v B(x,y)=0$.
On a sufficiently small ambient open neighborhood of $(x,y,0)$,
$P$ and $B$ are positive. At every point in this neighborhood
that belongs to the effective domain,
\[
 \Ecal_q^{(1)}(u,v,t)
 \geq\frac{P(u,t)}{\nuf(u,v)}+\lambda t^q
 \geq\frac{P(u,t)}{B(u,v)}+\lambda t^q=:\Phi(u,v,t).
\]
The function $\Phi$ is $C^1$ on this ambient neighborhood, including
negative $t$, since $q\geq2$ is an integer. Moreover,
\[
 \Phi(x,y,0)=r=\Ecal_q^{(1)}(x,y,0),\qquad
 \nabla\Phi(x,y,0)=
 \begin{pmatrix}
 \displaystyle\frac{A^*z+\nabla h(x)-rK^*y}{D}\\[1.5ex]
 0\\[1.5ex]0
 \end{pmatrix}.
\]
Outside the effective domain, including all points with $t<0$,
the potential is $+\infty$, so the lower-support inequality still
holds. Applying~\eqref{eq:support-remainder} and the normal-cone
inequality as in part (i) proves~\eqref{eq:E-zero-subgradient2}.
\end{proof}

The calmness assumption is unnecessary for the particular
Fr\'echet-subgradient inclusion in
Proposition~\ref{prop:E-subgradient2}, since the inclusion follows
from a direct smooth lower-support construction. In contrast,
the exact Fr\'echet-subdifferential identities in
\cite[Proposition~2.3]{LiShenZhangZhou2022} and
\cite[Lemma~2]{HTX} are established under calmness assumptions.

\medskip
\medskip
Define
\begin{eqnarray}\label{eq:lifted-states}
\tau_k:=\gamma_k^{1/q},\qquad
 W^{k+1}:=(x^{k+1},y^{k+1},\tau_k),\qquad
 Q_{k+1}:=\Ecal_q^{(1)}(W^{k+1}).
\end{eqnarray}
Theorem~\ref{PriTheo2R} ensures that these iterates $W^{k}$  belong to
$\Ccal_q^{(1)}$ for all sufficiently large $k$.

\begin{lemma}[Descent of the fixed potential]\label{lem:fixed-descent}
Suppose Assumption~\ref{ass:standing} holds. Let
$\{W^k\}_{k\geq1}$ be generated by Algorithm~\ref{alg:SFSPS}, and set
$
 \Omega_W:=\clust\{W^k\}_{k\geq1}.$
Then:
\begin{itemize}
\item[(i)] For all sufficiently large $k$,
\begin{eqnarray}\label{eq:Q-descent1}
 Q_{k+1}\leq Q_k-a\delta_ks_k^2,
 \qquad a:=\frac{\chi-1}{2\chi M}.
\end{eqnarray}
\item[(ii)]
$
 \lim_{k\to\infty}Q_k=\overline\theta,$
where $\overline\theta$ is defined in Theorem~\ref{PriTheo2R}.
\item[(iii)] For every
$(\overline x,\overline y,0)\in\Omega_W$,
$
 \Ecal_q^{(1)}(\overline x,\overline y,0)=\overline\theta.
$
\end{itemize}
\end{lemma}
\begin{proof}
For all sufficiently large $k$,
\[
 m\leq D_k\leq f(Kx^k)\leq M,
 \qquad \Psi_k(x^k)>0.
\]

(i) Set $r_{k+1}:=\Psi_{k+1}(x^{k+1})/D_{k+1}$.
Dividing~\eqref{eq:basic-residual-descent} by $D_{k+1}$ gives
\[
 r_{k+1}\leq\theta_k-a\delta_ks_k^2
                 +\lambda_0\Delta\gamma_k.
\]
Since $D_k\leq f(Kx^k)$ and $\Psi_k(x^k)>0$,
$\theta_k\leq r_k$. Consequently,
\[
 Q_{k+1}
 \leq Q_k-a\delta_ks_k^2
       -(\lambda-\lambda_0)\Delta\gamma_k
 \leq Q_k-a\delta_ks_k^2,
\]
which proves~\eqref{eq:Q-descent1}.

(ii) Since $D_k\leq M$, $\Gcal_q\geq\alpha/2$, and
$\lambda\geq0$,
$
 Q_k\geq\frac{\alpha}{2M}$
for all sufficiently large $k$. Telescoping
\eqref{eq:Q-descent1} therefore yields
$\sum_k\delta_ks_k^2<\infty$. Because
$\inf_k\delta_k>0$, we have $s_k\to0$.

By~\eqref{eq:Dhat} and the $\sigma_KB_y$-Lipschitz continuity of
$f\circ K$ on $\Sset$,
\[
 |D_k-f(Kx^k)|
 \leq2\sigma_KB_y s_{k-1}\longrightarrow0.
\]
Moreover, Lemma~\ref{lem:uniform-moreau}(iii)--(iv) and compactness
give a constant $B_\Psi>0$ such that
$|\Psi_k(x)|\leq B_\Psi$ for every $k\geq1$ and $x\in\Sset$.
Hence
\[
 |r_k-\theta_k|
 \leq
 \frac{2B_\Psi\sigma_KB_y}{m^2}s_{k-1}
 \longrightarrow0.
\]
Therefore
$Q_k=r_k+\lambda\gamma_{k-1}\to\overline\theta$.

(iii) Let $W^{k_j}\to(\overline x,\overline y,0)$. Since
$s_{k_j-1}\to0$, we also have $x^{k_j-1}\to\overline x$.
The relations
$y^{k_j}\in\partial f(Kx^{k_j-1})$ and the graph closedness of
$\partial f$ imply
\[
 \overline y\in\partial f(K\overline x),
 \qquad
 \nuf(\overline x,\overline y)=f(K\overline x).
\]
Furthermore, Lemma~\ref{lem:uniform-moreau}(iii)--(iv) gives
\[
 \left|\Psi_{k_j}(x^{k_j})
 -\bigl(g(A\overline x)+h(\overline x)\bigr)\right|
\leq\frac{\ell^2}{2}\gamma_{k_j-1}
 +|g(Ax^{k_j})-g(A\overline x)|
 +|h(x^{k_j})-h(\overline x)|
 \rightarrow0.
\]
Since $f\circ K$ is continuous and positive on $\Sset$,
$
 \theta_{k_j}
 =\frac{\Psi_{k_j}(x^{k_j})}{f(Kx^{k_j})}
 \rightarrow F(\overline x).
$
On the other hand, $\theta_k\to\overline\theta$, so
$F(\overline x)=\overline\theta$. Consequently,
$
 \Ecal_q^{(1)}(\overline x,\overline y,0)
 =F(\overline x)=\overline\theta.
$
\end{proof}

\begin{theorem}[Relative-error estimate]\label{lem:relative-error}
Suppose Assumption~\ref{ass:standing} holds. Let
$\{W^k\}_{k\geq1}$ be generated by Algorithm~\ref{alg:SFSPS}.
Then there exist $b>0$ and an index $N_3$ such that
\begin{eqnarray}\label{eq:corrected-relative-error}
 \dist\!\left(0,\partial\Ecal_q^{(1)}(W^{k+1})\right)
 \leq b\bigl(\delta_ks_k+\tau_k^{q-1}+\rho_k+\Delta\gamma_k\bigr)
 \qquad(k\geq N_3).
\end{eqnarray}
\end{theorem}
\begin{proof}
The projection step~\eqref{eq:algorithm-x} gives
\begin{eqnarray}\label{eq:projection-optimality}
 \xi^{k+1}:=-\nabla\Psi_k(x^k)+\theta_kK^*y^{k+1}-\delta_kd_k
 \in N_{\Sset}(x^{k+1}).
\end{eqnarray}
For all sufficiently large $k$, apply
Proposition~\ref{prop:E-subgradient2}(i) at $W^{k+1}$ with
$\zeta=Kx^k\in\partial f^*(y^{k+1})$ and
$n=\xi^{k+1}/D_{k+1}\in N_{\Sset}(x^{k+1})$.
 With
$r_{k+1}:=\Psi_{k+1}(x^{k+1})/D_{k+1}$, this gives
$w^{k+1}\in\whpartial\Ecal_q^{(1)}(W^{k+1})$ whose components are
\begin{align}
 w_x^{k+1}
 &=\frac{-\delta_kd_k+\nabla\Psi_{k+1}(x^{k+1})-\nabla\Psi_k(x^k)
              +(\theta_k-r_{k+1})K^*y^{k+1}}{D_{k+1}},
 \label{eq:x-component}\\
 w_y^{k+1}
 &=\frac{\Psi_{k+1}(x^{k+1})}{D_{k+1}^2}K(x^k-x^{k+1}),
 \label{eq:y-component}\\
 w_\tau^{k+1}
 &=q\tau_k^{q-1}\left(\lambda-
       \frac{\norm{\nabla g_{\gamma_k}(Ax^{k+1})}^2}{2D_{k+1}}\right).
 \label{eq:tau-component}
\end{align}
The gradients in~\eqref{eq:x-component} have different smoothing
parameters. Lemma~\ref{lem:moreau-parameter-gradient} and the
$1/\gamma_{k-1}$-Lipschitz continuity of $\nabla g_{\gamma_{k-1}}$ give
\begin{align}
\norm{\nabla(g_{\gamma_k}\circ A)(x^{k+1})
              -\nabla(g_{\gamma_{k-1}}\circ A)(x^k)}\leq\sigma_A\ell\rho_k
                 +\frac{\sigma_A^2}{\gamma_{k-1}}s_k.
 \label{eq:cross-parameter-gradient}
\end{align}
Adding the $h$-gradient difference and using
$\gamma_k\leq\gamma_{k-1}$ yields
\begin{eqnarray}\label{eq:cross-psi-gradient}
 \norm{\nabla\Psi_{k+1}(x^{k+1})-\nabla\Psi_k(x^k)}
 \leq\frac{\delta_k}{\chi}s_k+\sigma_A\ell\rho_k.
\end{eqnarray}

Set $L_G:=\sigma_A\ell+\max_{x\in\Sset}\norm{\nabla h(x)}$.
Lemma~\ref{lem:uniform-moreau}(i) and integration along segments
in the convex set $\Sset$ show that every $\Psi_j$ is
$L_G$-Lipschitz on $\Sset$. Moreover, there exists $B_G<\infty$ with
$\abs{\Psi_j(x)}\leq B_G$ for every $j\geq1$ and $x\in\Sset$,
by Lemma~\ref{lem:uniform-moreau}(iii)--(iv) and $\gamma_{j-1}\leq\gamma_0$.
The spatial Lipschitz bound and Lemma~\ref{lem:uniform-moreau}(v) imply
\begin{eqnarray}\label{eq:numerator-change}
 \abs{\Psi_{k+1}(x^{k+1})-\Psi_k(x^k)}
 \leq L_Gs_k+\frac{\ell^2}{2}\Delta\gamma_k.
\end{eqnarray}
Writing the ratio difference as
\[
 \theta_k-r_{k+1}
 =\frac{\Psi_k(x^k)-\Psi_{k+1}(x^{k+1})}{D_{k+1}}
 +\Psi_k(x^k)\frac{D_{k+1}-f(Kx^k)}{D_{k+1}f(Kx^k)}
\]
and using~\eqref{eq:numerator-change}, we obtain
\begin{eqnarray}\label{eq:theta-rhat}
 \abs{\theta_k-r_{k+1}}
 \leq\left(\frac{L_G}{m}+\frac{B_G\sigma_K B_y}{m^2}\right)s_k
                 +\frac{\ell^2}{2m}\Delta\gamma_k.
\end{eqnarray}
Since $\inf_k\delta_k>0$, combining~\eqref{eq:x-component},
\eqref{eq:cross-psi-gradient}, and~\eqref{eq:theta-rhat} yields
\[
 \norm{w_x^{k+1}}
 \leq C_x\bigl(\delta_ks_k+\rho_k+\Delta\gamma_k\bigr)
\]
for some constant $C_x>0$ independent of $k$.
 The remaining components satisfy
\[
 \norm{w_y^{k+1}}\leq\frac{B_G \sigma_K}{m^2}s_k,
 \qquad 0\leq w_\tau^{k+1}\leq q\lambda\tau_k^{q-1},
\]
where the second bound follows from
$\lambda\geq\lambda_0=\ell^2/(2m)$, $D_{k+1}\geq m$, and
$\|\nabla g_{\gamma_k}(Ax^{k+1})\|\leq\ell$.

Summing these bounds and using
$\whpartial\Ecal_q^{(1)}\subseteq\partial\Ecal_q^{(1)}$
proves~\eqref{eq:corrected-relative-error}.
\end{proof}

\subsection{The smooth-denominator potential}
Suppose that $f$ is differentiable on an open set containing
$K(\Sset)$ and has a gradient that is
$L_{\nabla f}$-Lipschitz continuous there.
Define
\begin{eqnarray}\label{eq:C2-definition}
\Ccal_q^{(2)}:=\Sset\times[0,\overline\tau],
\end{eqnarray}
and
\begin{eqnarray}\label{eq:E2-definition}
\Ecal_q^{(2)}(x,\tau):=
 \begin{cases}
 \displaystyle\frac{\Gcal_q(x,\tau)}{f(Kx)}+\lambda\tau^q,
           &(x,\tau)\in\Ccal_q^{(2)},\\[1ex]
 +\infty,&\text{otherwise}.
 \end{cases}
\end{eqnarray}
This potential omits the denominator dual variable.

\begin{proposition}[Estimates for the smooth-denominator potential]
\label{prop:smooth-potential}
Suppose Assumption~\ref{ass:standing} holds and $f$ has a Lipschitz
continuous gradient on an open set containing $K(\Sset)$.
\begin{itemize}
\item[(i)] The function $\Ecal_q^{(2)}$ is proper and lower semicontinuous.
\item[(ii)] For $x\in\Sset$, $0<\tau<\overline\tau$,
$D:=f(Kx)$, $r:=\Gcal_q(x,\tau)/D$,
$z_\tau:=\nabla g_{\tau^q}(Ax)$, and $n\in N_{\Sset}(x)$,
\begin{eqnarray}\label{eq:E2-positive-subgradient}
 \begin{pmatrix}
 \displaystyle\frac{A^*z_\tau+\nabla h(x)-rK^*\nabla f(Kx)}{D}+n\\[1.5ex]
 \displaystyle q\tau^{q-1}
           \left(\lambda-\frac{\norm{z_\tau}^2}{2D}\right)
 \end{pmatrix}
 \in\whpartial\Ecal_q^{(2)}(x,\tau).
\end{eqnarray}
\item[(iii)] At $\tau=0$, for $x\in\Sset$,
$z\in\partial g(Ax)$, and $n\in N_{\Sset}(x)$,
\begin{eqnarray}\label{eq:E2-zero-subgradient}
 \begin{pmatrix}
 \displaystyle\frac{A^*z+\nabla h(x)-F(x)K^*\nabla f(Kx)}{f(Kx)}+n\\[1.5ex]
 0
 \end{pmatrix}
 \in\whpartial\Ecal_q^{(2)}(x,0).
\end{eqnarray}
\end{itemize}
\end{proposition}
\begin{proof}
(i) Continuity of $\Gcal_q$ and positivity of $f\circ K$ on the compact
domain prove (i).\\
(ii) For $\tau>0$, differentiating the quotient
on an open neighborhood and adding the normal-cone term proves (ii).\\
(iii) At $\tau=0$, use the numerator lower support $P$ from the proof
of Proposition~\ref{prop:E-subgradient2}(ii). The function
\[
\widetilde\Phi:(u,t)\longmapsto\frac{P(u,t)}{f(Ku)}+\lambda t^q
\]
is $C^1$ on an open neighborhood of $(x,0)$ and is a smooth
lower support of $\Ecal_q^{(2)}$ at $V_0:=(x,0)$.
For $V=(u,t)$, since $\widetilde\Phi$ is $C^1$, its first-order remainder satisfies
\begin{eqnarray}\label{eq:Psi-remainder}
 R_{\widetilde\Phi}(V):=\widetilde\Phi(V)-\widetilde\Phi(V_0)
                 -\ip{\nabla\widetilde\Phi(V_0)}{V-V_0},
 \qquad
 \frac{R_{\widetilde\Phi}(V)}{\norm{V-V_0}}\longrightarrow0.
\end{eqnarray}
The gradient of $\widetilde\Phi$ yields~\eqref{eq:E2-zero-subgradient} with $n=0$.
Combining the signed-remainder argument~\eqref{eq:support-remainder}
with $\ip n{u-x}\leq0$ for $u\in\Sset$ proves (iii), including the
ambient perturbations with $t<0$.
\end{proof}

\medskip
\medskip

Define $\widetilde W^{k+1}:=(x^{k+1},\tau_k)$ and let
$
 P_{k+1}:=\Ecal_q^{(2)}(\widetilde W^{k+1})
         =\theta_{k+1}+\lambda\gamma_k.
$

\begin{lemma}[Descent of the fixed potential]\label{lem:smooth-fixed-descent}
Suppose that the assumptions of
Proposition~\ref{prop:smooth-potential} hold. Let
$\{\widetilde W^k\}_{k\geq1}$ be generated by
Algorithm~\ref{alg:SFSPS}.
For all sufficiently large $k$,
\begin{eqnarray}\label{eq:P-descent}
 P_{k+1}\leq P_k-a\delta_ks_k^2,\qquad a=\frac{\chi-1}{2\chi M}.
\end{eqnarray}
Consequently, $\lim\limits_{k \to +\infty} P_k=\overline \theta$, where $\overline\theta$ is defined in Theorem \ref{PriTheo2R}.
\end{lemma}
\begin{proof}
For all sufficiently large $k$,~\eqref{eq:theta-quasidescent} gives
\[
 \begin{aligned}
 P_{k+1}
 &\leq\theta_k-a\delta_ks_k^2
       +\lambda_0\Delta\gamma_k+\lambda\gamma_k\\
 &=P_k-a\delta_ks_k^2-(\lambda-\lambda_0)\Delta\gamma_k\\
 &\leq P_k-a\delta_ks_k^2,
 \end{aligned}
\]
because $\lambda\geq\lambda_0$ and $\Delta\gamma_k\geq0$.
Moreover, $P_k=\theta_k+\lambda\gamma_{k-1}\to\overline\theta$
by Theorem~\ref{PriTheo2R}(iv) and $\gamma_k\to0$.
\end{proof}
\begin{theorem}[Relative-error estimate for the smooth denominator]\label{residualDiffF}
Suppose that the assumptions of Proposition~\ref{prop:smooth-potential} hold.
Let $\{\widetilde W^k\}_{k\geq1}$ be generated by Algorithm \ref{alg:SFSPS}.
There exist $\widehat b>0$ and an index $\widehat N_3$ such that
\begin{eqnarray}\label{eq:E2-relative-error}
 \dist\!\left(0,\partial\Ecal_q^{(2)}(\widetilde W^{k+1})\right)
 \leq\widehat b\bigl(\delta_ks_k+\tau_k^{q-1}+\rho_k+\Delta\gamma_k\bigr)
 \qquad(k\geq \widehat N_3).
\end{eqnarray}
\end{theorem}
\begin{proof}
Use $n=\xi^{k+1}/f(Kx^{k+1})$ in
\eqref{eq:E2-positive-subgradient}, where $\xi^{k+1}$ is defined
in~\eqref{eq:projection-optimality}. Since
$y^{k+1}=\nabla f(Kx^k)$, the numerator of the resulting
$x$-component is
\[
\begin{aligned}
 &-\delta_kd_k+\nabla\Psi_{k+1}(x^{k+1})-\nabla\Psi_k(x^k)\\
 &\quad+(\theta_k-\theta_{k+1})K^*\nabla f(Kx^k)
 +\theta_{k+1}K^*
   \bigl(\nabla f(Kx^k)-\nabla f(Kx^{k+1})\bigr).
\end{aligned}
\]
The numerator estimate~\eqref{eq:numerator-change} and the
$\sigma_K B_y$-Lipschitz continuity of $f\circ K$ on $\Sset$ give
$\abs{\theta_{k+1}-\theta_k}\leq C(s_k+\Delta\gamma_k)$.
The final term is bounded in norm by
$
 \sup_j\abs{\theta_j}\,\sigma_K^2L_{\nabla f}s_k,
$
where $L_{\nabla f}$ is a Lipschitz constant of $\nabla f$ on
$K(\Sset)$. Using~\eqref{eq:cross-psi-gradient},
$f(Kx^{k+1})\geq m$, and $\inf_k\delta_k>0$, we obtain
\[
 \norm{w_x^{k+1}}
 \leq C'\bigl(\delta_ks_k+\rho_k+\Delta\gamma_k\bigr),
\]
where $C'>0$ is independent of $k$.
The $\tau$-component lies between zero and
$q\lambda\tau_k^{q-1}$, as in~\eqref{eq:tau-component}.
This proves~\eqref{eq:E2-relative-error}.
\end{proof}

\section{Whole-sequence convergence}\label{sec:main}
The following theorem gives a scaled KL finite-length principle for
variable-coefficient descent systems satisfying suitable descent and
relative-error estimates.

\begin{theorem}\label{lem:scaled-KL}
Let $H:\R^d\to(-\infty,+\infty]$ be a proper lower semicontinuous
function, let $\{U^k\}_{k\ge 1}\subset\R^d$ be bounded, and write
$H_k:=H(U^k)$. Suppose that $H_k\to H_\infty\in\R$ and that
$\Omega:=\clust\{U^k\}\subseteq\dom\partial H$ satisfies
$H\equiv H_\infty$ on $\Omega$. Assume that $H$ has the KL property
on $\Omega$. Let $s_k\geq0$, and let $\{\delta_k\}$ be a
nondecreasing positive scalar sequence. Suppose that, for all
sufficiently large $k$,
\begin{eqnarray*}
\begin{aligned}
\textnormal{(a)}\quad&H_k-H_{k+1}\geq a\delta_ks_k^2,\\
\textnormal{(b)}\quad&\dist(0,\partial H(U^{k+1}))
 \leq b(\delta_ks_k+\varepsilon_k),
\end{aligned}
\end{eqnarray*}
where $a,b>0$ and $\varepsilon_k\geq0$, with
\begin{eqnarray}\label{eq:abstract-error-sum}
\sum_{k=1}^{+\infty}\frac{\varepsilon_k}{\delta_{k+1}}<\infty.
\end{eqnarray}
Then $\sum_{k=1}^{+\infty} s_k<\infty$.
\end{theorem}
\begin{proof}
A bounded sequence approaches its compact cluster set. Thus, Lemma \ref{lem:uniform-KL} provides a  concave desingularizing function \(\varphi\) such that, for all sufficiently large \(k\) satisfying
$H_{k+1}>H_\infty$,
\begin{eqnarray}\label{eq:uniform-KL}
 \varphi'(H_{k+1}-H_\infty)
 \dist\!\left(0,\partial H(U^{k+1})\right)\geq1.
\end{eqnarray}

Suppose first that $H_j=H_\infty$ for some sufficiently large $j$.
Monotonicity and convergence then give
$H_k=H_\infty$ for every $k\geq j$, and the descent estimate implies
$s_k=0$ for every $k\geq j$. The conclusion follows in this case.

It remains to consider the case in which $H_k>H_\infty$ when $k$ is sufficiently large. Define
\[
 \Delta_{k+1}:=
 \varphi(H_{k+1}-H_\infty)
 -\varphi(H_{k+2}-H_\infty)\geq0.
\]
Concavity of $\varphi$, the descent estimate at index $k+1$,
\eqref{eq:uniform-KL}, and the relative-error estimate at index $k$
yield
\begin{eqnarray}\label{eq:Delta-lower}
\begin{aligned}
 \Delta_{k+1}
 &\geq\varphi'(H_{k+1}-H_\infty)(H_{k+1}-H_{k+2})\\
 &\geq\frac{a\delta_{k+1}s_{k+1}^2}
 {b(\delta_ks_k+\varepsilon_k)}.
\end{aligned}
\end{eqnarray}
The denominator in~\eqref{eq:Delta-lower} is positive. Indeed, if
$\delta_ks_k+\varepsilon_k=0$, then the relative-error estimate gives
$\dist(0,\partial H(U^{k+1}))=0$, contradicting
\eqref{eq:uniform-KL}.

Set $C:=b/a$. Since $\delta_k\leq\delta_{k+1}$,
$$
 s_{k+1}^2
 \leq C\left(
 \frac{\delta_k}{\delta_{k+1}}s_k
 +\frac{\varepsilon_k}{\delta_{k+1}}\right)\Delta_{k+1}
 \leq C\left(s_k+\frac{\varepsilon_k}{\delta_{k+1}}\right)
 \Delta_{k+1}.$$
The inequality $2\sqrt{uv}\leq u+v$ therefore gives
\begin{eqnarray}\label{eq:length-recursion}
 2s_{k+1}\leq s_k+\frac{\varepsilon_k}{\delta_{k+1}}
 +C\Delta_{k+1}.
\end{eqnarray}
Summing~\eqref{eq:length-recursion} from a sufficiently large index
$N$ to $T\geq N$ and telescoping the terms $\Delta_{k+1}$ yield
\begin{eqnarray*}
 \sum_{k=N}^{T}s_{k+1}+s_{T+1}
 \leq s_N+\sum_{k=N}^{T}\frac{\varepsilon_k}{\delta_{k+1}}
 +C\varphi(H_{N+1}-H_\infty).
\end{eqnarray*}
Letting $T\to\infty$ and using~\eqref{eq:abstract-error-sum}, we obtain
$\sum_{k=N}^{\infty}s_{k+1}<\infty$. Adding the finite initial
segment completes the proof.
\end{proof}

\begin{remark}
Theorem~\ref{lem:scaled-KL} provides a scaled finite-length variant
of the standard KL argument, motivated by variable-coefficient
KL frameworks \cite{FrankelGarrigosPeypouquet2015}, which already
allow unbounded descent coefficients. Here, it converts the
descent and relative-error estimates for the fixed lifted
S-FSPS potentials into finite primal length.
\end{remark}

\begin{theorem}[Whole-sequence convergence of the primal iterates]
\label{thm:main}
Suppose Assumption~\ref{ass:standing} holds, and let
Algorithm~\ref{alg:SFSPS} use smoothing parameters
$\{\gamma_k\}_{k\geq0}$ satisfying~\eqref{eq:gamma-basic} and
descent coefficients $\{\delta_k\}_{k\geq0}$ defined by
\eqref{eq:delta}.
Fix an integer $q\geq2$ and assume that
\begin{eqnarray}\label{eq:schedule-condition}
\sum_{k=1}^{\infty}
 \frac{\gamma_k^{1-1/q}}{\delta_{k+1}}<\infty.
\end{eqnarray}
Consider either of the following lifted constructions:
\begin{enumerate}[label=(\alph*),leftmargin=2.2em]
\item Let
\[
 E=\Ecal_q^{(1)},\qquad
 U^{k+1}=(x^{k+1},y^{k+1},\gamma_k^{1/q}).
\]

\item Assume that $f$ is differentiable with a Lipschitz continuous
gradient on an open set containing $K(\Sset)$, and let
\[
 E=\Ecal_q^{(2)},\qquad
 U^{k+1}=(x^{k+1},\gamma_k^{1/q}).
\]
\end{enumerate}
Let
\[
 \Omega:=\clust\{U^{k+1}\}_{k\geq1},
\]
and assume that $E$ has the KL property at every point of $\Omega$.
Then:
\begin{enumerate}
\item[(i)] $\Omega$ is nonempty and compact,
$\Omega\subseteq\dom\partial E$, and
\begin{eqnarray}\label{eq:cluster-value}
 E(\overline U)=\overline\theta
 \qquad\text{for every }\overline U\in\Omega,
\end{eqnarray}
where $\overline\theta$ is defined in Theorem~\ref{PriTheo2R}.
\item[(ii)] The primal sequence has finite length:
\begin{eqnarray}\label{eq:finite-length}
\sum_{k=0}^{\infty}\norm{x^{k+1}-x^k}<\infty.
\end{eqnarray}
Consequently, it converges to an exact limiting lifted stationary point
$\overline x$ of~\eqref{eq:problem}.
\end{enumerate}
\end{theorem}

\begin{proof}
(i) Theorem~\ref{PriTheo2R} and $\gamma_k\to0$ imply that the lifted
sequence is bounded and that its cluster set $\Omega$ is nonempty
and compact. We first verify the cluster-value and subdifferential
conditions required by Theorem~\ref{lem:scaled-KL}.

In case (a), let
$U^{k_j+1}\to\overline U=(\overline x,\overline y,0)$.
Since $s_{k_j}\to0$, we also have $x^{k_j}\to\overline x$.
The inclusion $y^{k_j+1}\in\partial f(Kx^{k_j})$ and graph
closedness give
\[
 \overline y\in\partial f(K\overline x),\qquad
 \nuf(\overline x,\overline y)=f(K\overline x).
\]
Lemma~\ref{lem:fixed-descent}(iii) therefore gives
$E(\overline U)=\overline\theta$.
Assumption~\ref{ass:standing}(f) supplies
$\overline z\in\partial g(A\overline x)$.
Proposition~\ref{prop:E-subgradient2}(ii), with $n=0$, yields
$\whpartial E(\overline U)\neq\varnothing$, and hence
$\overline U\in\dom\partial E$.
In case (b), every cluster point has the form $(\overline x,0)$.
Continuity of $\Gcal_q$ on $\Sset\times[0,\overline\tau]$,
positivity and continuity of $f\circ K$, and
Lemma~\ref{lem:smooth-fixed-descent} give
\[
 E(\overline x,0)
 =\lim_{j\to+\infty} E(U^{k_j+1})=\overline\theta
\]
along any subsequence converging to $(\overline x,0)$.
Choose $\overline z\in\partial g(A\overline x)$ as above.
Proposition~\ref{prop:smooth-potential}(iii), with $n=0$, yields
$(\overline x,0)\in\dom\partial E$.
This proves part (i) in both cases.

(ii) We discard finitely many initial iterations 
so that the lifted iterates lie in the effective domain of
 the corresponding potential and all descent and relative-error estimates hold.
The potentials are proper and lower semicontinuous by
Propositions~\ref{prop:E-subgradient} and
\ref{prop:smooth-potential}(i), respectively, and have the KL property
on $\Omega$ by assumption. Their values along the iterates converge
to $\overline\theta$ by Lemmas~\ref{lem:fixed-descent}(ii) and
\ref{lem:smooth-fixed-descent}.
The descent and relative-error estimates required by
Theorem~\ref{lem:scaled-KL} follow from
Lemma~\ref{lem:fixed-descent}(i) and
Theorem~\ref{lem:relative-error} in case (a), and from
Lemma~\ref{lem:smooth-fixed-descent} and
Theorem~\ref{residualDiffF} in case (b).
The sequence $\delta_k$ is positive and nondecreasing
by~\eqref{eq:delta}.
It remains to verify~\eqref{eq:abstract-error-sum} with
$
 \varepsilon_k:=\rho_k+\gamma_k^{1-1/q}+\Delta\gamma_k.$

We note that
\begin{eqnarray*}
\sum_{k=1}^{\infty}\frac{\rho_k}{\delta_{k+1}}\leq\frac1{\chi\sigma_A^2}
       \sum_{k=1}^{\infty}
       \frac{\gamma_{k+1}}{\gamma_{k-1}}\Delta\gamma_k\leq\frac1{\chi\sigma_A^2}
       \sum_{k=1}^{\infty}\Delta\gamma_k
 =\frac{\gamma_0}{\chi\sigma_A^2}<\infty.
\end{eqnarray*}
Also, $\underline\delta:=\inf_k\delta_k>0$ implies
$
 \sum_{k=1}^{\infty}\frac{\Delta\gamma_k}{\delta_{k+1}}
 \leq\frac{\gamma_0}{\underline\delta}<\infty.
$
Together with~\eqref{eq:schedule-condition}, this verifies
\eqref{eq:abstract-error-sum}.
Theorem~\ref{lem:scaled-KL} gives
$
 \sum_{k=1}^{\infty}s_k<\infty.
$
Since $s_0<\infty$, this proves~\eqref{eq:finite-length}, and hence
$x^k\to\overline x\in\Sset$.
By Theorem~\ref{PriTheo2R}(v), there is a subsequence with
$\delta_{k_j}s_{k_j}\to0$. Theorem~\ref{PriTheo2R}(vii) then shows
that $\overline x$ is an exact limiting lifted stationary point.
\end{proof}

\subsection{Definability and smoothing schedules}\label{sec:consequences}

\begin{proposition}[Definability and the KL property]
\label{prop:definable}
Suppose that Assumption~\ref{ass:standing} holds and that
$\Sset$, $f$, $g$, and $h$ are definable in a common o-minimal
expansion of the real field.
Fix an integer $q\geq2$, and let $\Ecal_q^{(1)}$ and
$\Ecal_q^{(2)}$ be
defined in~(\ref{eq:E-definition}) and~(\ref{eq:E2-definition}),
respectively. Then $f^*$, the jointly parameterized Moreau
envelope $g_{\gamma}(u)$, the domains $\Ccal_q^{(1)}$ and
$\Ccal_q^{(2)}$, and the potentials $\Ecal_q^{(1)}$ and
$\Ecal_q^{(2)}$ are definable. Moreover, both potentials
have the KL property throughout their limiting-subdifferential
domains, with concave desingularizing functions. In particular,
these conclusions hold for semialgebraic data satisfying
Assumption~\ref{ass:standing}.
\end{proposition}

\begin{proof}
Since
\[
 \epi f^*
 =\bigl\{(y,t):
     \langle u,y\rangle-f(u)\leq t
     \text{ for every }u\in\dom f\bigr\},
\]
the epigraph of $f^*$ is definable. Its finite-valued graph, and hence
the extended-real-valued function $f^*$, are therefore definable.

For $\gamma>0$, recall that
\[
 g_\gamma(u)
 :=\min_{v\in\R^s}
 \left\{g(v)+\frac{\|u-v\|^2}{2\gamma}\right\}.
\]
Because $g$ is proper, convex, and lower semicontinuous,
the minimum in this definition is finite and attained. It follows that
\begin{align*}
 \epi\bigl((u,\gamma)\mapsto g_\gamma(u)\bigr)=
 \left\{(u,\gamma,t):
   \gamma>0,\ \exists v\in\dom g,\quad
   g(v)+\frac{\|u-v\|^2}{2\gamma}\leq t
 \right\}.
\end{align*}
This set is definable, so the jointly parameterized Moreau envelope is
definable on $\R^s\times(0,+\infty)$.

The map $\tau\mapsto\tau^q$ is polynomial. Composition with this map
and the linear operator $A$, followed by adjoining the branch at
$\tau=0$, shows that
\[
 \Gcal_q(x,\tau)
 =\begin{cases}
   g_{\tau^q}(Ax)+h(x),&\tau>0,\\
   g(Ax)+h(x),&\tau=0,
  \end{cases}
\]
is definable for $\tau\geq0$. Here adjoining the zero branch uses closure
under finite unions of definable graphs.

The original denominator $f\circ K$, the Fenchel denominator
\[
 \nuf(x,y)=\langle Kx,y\rangle-f^*(y),
 \qquad y\in\dom f^*,
\]
and the constraints defining $\Ccal_q^{(1)}$ and $\Ccal_q^{(2)}$ are
definable. On each effective domain, the corresponding denominator is
positive; hence division by the denominator and addition of the
polynomial term $\lambda\tau^q$ preserve definability. We regard the
potentials as extended-real-valued functions and use the standard
convention that such a function is definable when its epigraph is
definable~\cite{BolteDaniilidisLewisShiota2007}. Extending the
restrictions by $+\infty$ outside the definable effective domains
therefore shows that both potentials are definable.
Let $E$ denote either
potential and fix $\overline U\in\dom\partial E$, where $\partial E$
denotes the limiting subdifferential. Write $\partial^\circ E$ for the
Clarke subdifferential. Applying the nonsmooth definable inequality of
\cite[Theorem~14]{BolteDaniilidisLewisShiota2007} to
$E-E(\overline U)$ gives the corresponding definable functions
$\psi$ and $\omega$. Choose a bounded neighborhood $\cal V$ of
$\overline U$ with compact closure. Since $\omega$ is continuous and
strictly positive, after shrinking $\cal V$ and choosing $\eta>0$
small enough so that
$
 \eta<\inf_{U\in\overline{\cal V}}\omega(\|U\|),
$
we obtain a continuous definable function
$\psi:[0,\eta)\to[0,+\infty)$ with $\psi(0)=0$,
$\psi\in C^1(0,\eta)$, and $\psi'>0$, such that
\begin{eqnarray}\label{eq:Clarke-KL-definable}
 \psi'\bigl(E(U)-E(\overline U)\bigr)
 \dist\bigl(0,\partial^\circ E(U)\bigr)\geq1
\end{eqnarray}
whenever $U\in {\cal V}$ and
$0<E(U)-E(\overline U)<\eta$.

The inequality in~\eqref{eq:Clarke-KL-definable} is stated in terms of
the Clarke subdifferential, whereas Definition~\ref{def:KL} uses the
limiting subdifferential. For every proper lower semicontinuous
function,
\[
 \partial E(U)\subseteq\partial^\circ E(U).
\]
Consequently,
$$
 \dist\bigl(0,\partial E(U)\bigr)
 \geq
 \dist\bigl(0,\partial^\circ E(U)\bigr).
$$
Thus~\eqref{eq:Clarke-KL-definable} implies
$$
 \psi'\bigl(E(U)-E(\overline U)\bigr)
 \dist\bigl(0,\partial E(U)\bigr)\geq1,
$$
which is the KL inequality
in Definition~\ref{def:KL}.
Since $\psi'$ is definable, the o-minimal monotonicity theorem
\cite[Chapter~3, Section~1]{vanDenDries1998} implies that, after
shrinking $\eta$ if necessary, $\psi'$ is monotone on $(0,\eta)$.
If $\psi'$ is nonincreasing (which includes the strictly decreasing
and constant cases), then $\psi$ is concave on $[0,\eta)$.
If $\psi'$ is strictly increasing, choose any $a\in(0,\eta)$ and set
\[
   \varphi(s):=\psi'(a)s,\qquad 0\le s<a.
\]
Then $\varphi$ is linear and hence concave, and
\[
   \varphi'(s)=\psi'(a)\geq \psi'(s)>0,
   \qquad 0<s<a.
\]
Therefore, the KL inequality with a concave desingularizing function holds at every point of \(\operatorname{dom}\partial E\), including boundary points with \(\tau=0\) whenever such points belong to \(\operatorname{dom}\partial E\). For the zero-smoothing cluster points arising in the convergence analysis, this membership follows separately from Proposition~\ref{prop:E-subgradient2}(ii) in the Fenchel case and Proposition~\ref{prop:smooth-potential}(iii) in the smooth-denominator case.

Finally, semialgebraic sets and functions are definable in the
o-minimal structure of the real field. The semialgebraic case follows.
\end{proof}

\begin{corollary}\label{cor:definable-main}
Suppose Assumption~\ref{ass:standing} holds and that
$\Sset$, $f$, $g$, and $h$ are definable in a common o-minimal
expansion of the real field. Let $\{\gamma_k\}_{k\geq0}$
satisfy~\eqref{eq:gamma-basic}, and fix an integer $q\geq2$.
\begin{itemize}
\item[(i)] If~\eqref{eq:schedule-condition} holds, then the primal sequence generated by Algorithm \ref{alg:SFSPS} has finite length and converges to an exact limiting lifted stationary point.
\item[(ii)] If
$ \sum_{k=0}^{\infty}\gamma_k^{2-1/q}<\infty,
$
then~\eqref{eq:schedule-condition} holds.
\item[(iii)] Let the smoothing parameters be defined by
\begin{eqnarray}\label{eq:power-schedule}
 \gamma_k=(k+k_0)^{-\beta},\qquad k_0\geq1,\quad 0<\beta\leq1.
\end{eqnarray}
Then~\eqref{eq:schedule-condition} holds if and only if
\begin{eqnarray}\label{eq:p-threshold}
 \beta>\frac{q}{2q-1}.
\end{eqnarray}
Consequently, when $\{\gamma_k\}$ is defined by~\eqref{eq:power-schedule} and
$\frac{q}{2q-1}<\beta\leq1$,
all
conclusions  in part~{\rm(i)} hold.
\end{itemize}
\end{corollary}
\begin{proof}
(i) Proposition~\ref{prop:definable} gives the KL property of
$\Ecal_q^{(1)}$ on $\dom\partial\Ecal_q^{(1)}$.
Part~(i) therefore follows from Theorem~\ref{thm:main}.\\
(ii) Monotonicity of $\{\gamma_k\}$ and~\eqref{eq:delta} give
\[
 \frac{\gamma_k^{1-1/q}}{\delta_{k+1}}
 \leq\frac{\gamma_k^{1-1/q}\gamma_{k+1}}{\chi\sigma_A^2}
 \leq\frac{\gamma_k^{2-1/q}}{\chi\sigma_A^2}.
\]
Thus~$ \sum_{k=0}^{\infty}\gamma_k^{2-1/q}<\infty$ implies~\eqref{eq:schedule-condition}.\\
(iii) For~\eqref{eq:power-schedule}, $\delta_{k+1}$ is bounded above and
below by positive constant multiples of $(k+k_0)^\beta$ for all
sufficiently large $k$. Hence
$\gamma_k^{1-1/q}/\delta_{k+1}$ is bounded above and below by
positive constant multiples of $(k+k_0)^{-\beta(2-1/q)}$.
The corresponding series converges exactly when
$\beta(2-1/q)>1$, which is equivalent to~\eqref{eq:p-threshold}.
Since $0<\beta\leq1$, the power schedule also satisfies the
nonsummability condition in~\eqref{eq:gamma-basic}. This proves
part~(iii).
\end{proof}

\begin{remark}
For every $1/2<\beta\leq1$, one can choose an integer $q\geq2$
satisfying $q>\beta/(2\beta-1)$. The resulting power schedule then
lies in the convergence range of Corollary~\ref{cor:definable-main}.
\end{remark}

\begin{corollary}[The square-root lift]\label{cor:square-root}
Under the hypotheses of Corollary~\ref{cor:definable-main}, choosing
$q=2$ gives the sufficient condition $\sum_k\gamma_k^{3/2}<\infty$.
For power schedules, this covers $2/3<\beta\leq1$, including the
harmonic schedule $\beta=1$.
\end{corollary}

\begin{table}[htbp]
\centering
\caption{Sufficient conditions for finite primal length under power smoothing schedules.}\label{tab:schedules}
\begin{tabular}{ccc}
\toprule
Lift & Sufficient series & Power range\\
\midrule
$\gamma=\tau^2$ & $\sum_k\gamma_k^{3/2}<\infty$ & $2/3<\beta\leq1$\\
$\gamma=\tau^3$ & $\sum_k\gamma_k^{5/3}<\infty$ & $3/5<\beta\leq1$\\
$\gamma=\tau^4$ & $\sum_k\gamma_k^{7/4}<\infty$ & $4/7<\beta\leq1$\\
$\gamma=\tau^q$ & $\sum_k\gamma_k^{2-1/q}<\infty$
 & $q/(2q-1)<\beta\leq1$\\
\bottomrule
\end{tabular}
\end{table}

\section{Composite specialization and examples}
\label{sec:examples}
\subsection{Nonconvex composite optimization}

\begin{theorem}
\label{cor:composite}
Consider the problem
\begin{eqnarray}\label{comp}
 \min_{x\in\Sset}g(Ax)+h(x)
\end{eqnarray}
under Assumption~\ref{ass:standing}(a)--(c) and (f).
Let $A\neq0$. Suppose that
$\Sset$, $g$, and $h$ are definable in a common o-minimal expansion
of the real field. Let $x^0\in\Sset$, $z^0\in\R^s$, and
\[
 \gamma_k=(k+k_0)^{-\beta},
 \qquad k_0\geq1,\qquad \frac12<\beta\leq1.
\]
Let $\delta_k$ be defined by~\eqref{eq:delta}, and consider the iteration
\begin{eqnarray}\label{eq:composite-iteration}
 x^{k+1}=\proj_{\Sset}\left(
 x^k-\frac{A^*z^k+\nabla h(x^k)}{\delta_k}\right),
 \qquad
 z^{k+1}=\nabla g_{\gamma_k}(Ax^{k+1}).
\end{eqnarray}
The primal sequence $\{x^k\}_{k\ge 0}$ generated by
\eqref{eq:composite-iteration} has finite length and converges to an exact limiting stationary point $\overline x$
satisfying
\begin{eqnarray}\label{eq:composite-stationarity}
 0\in A^*\partial g(A\overline x)
      +\nabla h(\overline x)+N_{\Sset}(\overline x).
\end{eqnarray}
\end{theorem}
\begin{proof}
Assumption~\ref{ass:standing}(f) implies that $g(Ax)$ is finite on
$\Sset$. Hence $g\circ A+h$ is lower semicontinuous and finite on the
compact set $\Sset$, so
\[
    \inf_{x\in\Sset}\bigl(g(Ax)+h(x)\bigr)>-\infty.
\]
Choose a constant $c$ such that
$g(Ax)+h(x)+c>0$ for every $x\in\Sset$, and set
$\widetilde h:=h+c$. Since $\nabla\widetilde h=\nabla h$, this
shift changes neither the iteration~\eqref{eq:composite-iteration}
nor the stationarity condition~\eqref{eq:composite-stationarity}.
Now set $f\equiv1$ and $K=I$. Then $\partial f(Kx)=\{0\}$, and
S-FSPS for the shifted fractional problem reduces exactly to
\eqref{eq:composite-iteration}.

Because $\beta>1/2$, choose an integer $q\geq2$ such that
$
 q>\frac{\beta}{2\beta-1},
$
which is equivalent to $\beta>q/(2q-1)$. The shifted data remain
definable, and Corollary~\ref{cor:definable-main}(iii) therefore gives
finite primal length and convergence to an
exact stationary point of the composite problem (\ref{comp}). For $f\equiv1$, the lifted-stationarity condition
reduces precisely to~\eqref{eq:composite-stationarity}.
\end{proof}
\begin{remark}
Assumption~\ref{ass:standing}(e)  is unnecessary in the composite setting, as shown by the constant-shift argument in the proof of Theorem \ref{cor:composite}.
\end{remark}
\subsection{Examples and limitations}
\label{subsec:examples-limitations}
\begin{example}
\label{ex:active}
Let
\begin{eqnarray}\label{Example6.2}
 \Sset=[-1,1]^2,\quad
 A=\begin{pmatrix}1&0\\2&0\end{pmatrix},\quad
 K=\begin{pmatrix}0&1\end{pmatrix},\quad
 g(u)=\abs{u_1}+\abs{u_2},
\end{eqnarray}
and let $f(v)=2+\abs v$ and $h(x)=1+x_2^2/2$. Then
\[
 F(x)=\frac{3\abs{x_1}+1+x_2^2/2}{2+\abs{x_2}}.
\]
All data are semialgebraic. Assumption~\ref{ass:standing} holds with
$L_{\nabla h}=1$, $\alpha=1$, $\sigma_A^2=5$, $\ell=\sqrt2$,
$m=1$, $M=3$, and $B_y=1$. The set $\Sset$ is full-dimensional,
both $g\circ A$ and $f\circ K$ are nonconstant, and $A$ has rank one
with no zero row.
Since the data are semialgebraic and satisfy Assumption 2.1, Corollary~\ref{cor:definable-main} applies for every admissible denominator-subgradient selection, provided that the smoothing schedule satisfies its hypotheses. A direct calculation shows that the
problem (\ref{eq:problem}) with the data specified in
(\ref{Example6.2}) has three limiting lifted stationary points:
$
x^{(1)}=(0,0),\; x^{(2)}=(0,\sqrt{6}-2),\;\mbox{and}\; x^{(3)}=(0,2-\sqrt{6}).
$

However, \(x^{(1)}=(0,0)\) is not a limiting stationary point of
problem~\eqref{eq:problem} for the data specified in
Example~\ref{ex:active}, since
\[
 \partial(F+\iota_{\Sset})(0,0)
 =\left[-\frac32,\frac32\right]
       \times\left\{-\frac14,\frac14\right\}
 \not\ni(0,0).
\]
Moreover, initialization with $x^0=(0,0)$, $z^0=0$, and
$\theta_0=1/2$, followed by the admissible selections
$y^{k+1}=0$, gives $x^k=(0,0)$ and $z^k=0$ for all $k$.
Thus the exact limiting lifted stationarity conclusion of
Theorem~\ref{thm:main} cannot in general be strengthened to
$0\in\partial(F+\iota_{\Sset})(\overline x)$.

\end{example}

\begin{remark}\label{rem:whole-sequence}
The convergence results concern the primal sequence $\{x^k\}$
and the quotient values $\{\theta_k\}$. If
$\partial f(K\overline x)$ is a singleton, boundedness and graph
closedness also give convergence of $y^{k+1}$. If
$\partial g(A\overline x)$ is a singleton, the identity
$z^k\in\partial g(Ax^k-\gamma_{k-1}z^k)$ gives convergence of $z^k$.
\end{remark}

The following examples clarify two aspects of the convergence results.
First, convergence  of the primal sequence does not
ensure convergence of the dual sequence. Second, nonsummability of
the smoothing parameters cannot be omitted from a general
stationarity guarantee under the stated assumptions.

\begin{example}
\label{ex:y-oscillation}
Let $\Sset=\{0\}\subset\R$, $A=K=1$,
$g(u)=1+\abs u$, $f(v)=1+\abs v$, and $h=0$.
Initialize $x^0=z^0=0$ and $\theta_0=1$.
Projection yields $x^k=0$ for every $k\geq0$.
Since
\[
 g^*(z)=-1+\ind_{[-1,1]}(z),
 \qquad g_\gamma(0)=1 \quad(\gamma>0),
\]
the dual and quotient updates yield $z^k=0$ and $\theta_k=1$.
On the other hand, $\partial f(0)=[-1,1]$, so
\[
 y^{k+1}=\frac12(-1)^{k+1}
\]
is an admissible, nonconvergent sequence of dual selections.
Moreover, these selections lie in
$\operatorname{int}(\dom f^*)=(-1,1)$.
Thus, even for semialgebraic data and a constant primal sequence,
the full primal--dual sequence need not converge.
\end{example}

\begin{example}\label{ex:summable}
Let $\Sset=[0,1]$, $A=K=1$, $g(u)=u$, $h(x)=1$, $f(v)=1$,
and $\chi=2$. Choose $x^0=z^0=1$, any $\theta_0>0$, and
$\gamma_k=2^{-(k+2)}$. The standing  assumptions hold with
$L_{\nabla h}=0$, $\ell=1$, and $\alpha=1$; the smoothing sequence
violates the nonsummability condition~\eqref{eq:gamma-basic}.
The dual updates give $z^k=1$ and $y^{k+1}=0$, so
$
 x^{k+1}=x^k-\frac{\gamma_k}{2},$ and
 $x^k=1-\frac12\sum_{j=0}^{k-1}\gamma_j\rightarrow\frac{3}{4}$.
All iterates lie in $[3/4,1]$, so projection does not alter this formula.
We note that
$0\notin1+N_{\Sset}(3/4)$. This shows that the stationarity conclusion
in Theorem~\ref{thm:main} cannot  be guaranteed without the nonsummability condition ~\eqref{eq:gamma-basic}.
\end{example}

\subsection{Relation to the earlier
counterexamples}
\label{sec:relationship-counterexamples}
The constructions in~\cite{Tao2026} produce a slowly rotating primal
sequence satisfying
\begin{eqnarray}\label{inftynew}
 \clust\{x^k\}=\{1\}\times\mathbb S^1,
 \qquad
 \sum_k\norm{x^{k+1}-x^k}=+\infty,
\end{eqnarray}
although every cluster point is an exact limiting lifted stationary
point. The smoothing schedule satisfies
$\sum_k\gamma_k=+\infty$ and $\sum_k\gamma_k^{3/2}<\infty$,
and thus meets the schedule requirements for the square-root lift
$q=2$. For the basic construction, where $f\equiv1$, the applicable
potential is $\Ecal_2^{(2)}$. If this potential had the KL property
at every lifted cluster point, Theorem~\ref{thm:main} would imply
$\sum_k\norm{x^{k+1}-x^k}<\infty$, contradicting (\ref{inftynew}). Hence $\Ecal_2^{(2)}$ must fail to have the
KL property at one or more lifted stationary points. 

\section{Conclusion}\label{sec:conclusion}
Under a KL assumption on a fixed lifted potential and suitable
smoothing-schedule conditions, we have established whole-sequence
convergence of the primal iterates of S-FSPS to an exact limiting
lifted stationary point. The
result is obtained by constructing fixed lifted potentials and developing
a scaled KL finite-length analysis framework that accommodates the changing
smoothing parameters, including the unbounded descent coefficients
induced by vanishing smoothing. Our principal contribution is the
construction of these potentials together with compatible
sufficient-decrease and relative-error estimates along the original
variable-smoothing trajectory. Theorem~\ref{lem:scaled-KL} supplies the scaled finite-length
argument needed to convert these estimates into finite primal length.
After scaling, the residual generated by changes in the Moreau
parameter is controlled through a telescoping argument, while the
remaining smoothing contribution is handled by a fixed power lift.
Under the standing assumptions, the nonsummability of the smoothing
parameters ensures that the primal sequence converges to an exact limiting
lifted stationary point.

For the Fenchel potential, the subgradient construction via a smooth
lower support does not require calmness of $f^*$, while the
smooth-denominator potential uses a lower-dimensional lifted space.
For definable data, the convergence result covers power schedules
$\gamma_k=(k+k_0)^{-\beta}$ with $1/2<\beta\leq1$, provided that the
lift exponent is chosen suitably. The examples show that finite
primal length does not ensure convergence of the dual sequence and
that summable smoothing can lead to a nonstationary primal limit.
Future directions include extending the analysis to power schedules
with $0<\beta\leq1/2$ and to more general smoothing schedules.

\end{document}